\documentclass[oneside]{amsart} 
\usepackage{amsmath}
\usepackage[inner=2.5cm,outer=25mm, top=25mm, bottom=25mm]{geometry}
\usepackage[latin1]{inputenc}
\usepackage{fancyhdr}
\usepackage{hyperref}
\usepackage{stmaryrd}
\usepackage{epsfig}
\usepackage{tikz}
\usepackage{euscript}
\usepackage{color}
\usepackage{pgfplots}
\usetikzlibrary{calc}       		
\usepackage[scaled=0.9]{helvet}		
\usepackage{courier}			
\usepackage{bbm}
\usepackage[latin1]{inputenc}
\usepackage{listings}
\usepackage{graphicx}
\usepackage{amsmath}
\usepackage{amssymb}
\usepackage{bbm}
\newtheorem{theorem}{Theorem}

\newtheorem{corollary}[theorem]{Corollary}
\newtheorem{definition}[theorem]{Definition}
\newtheorem{example}[theorem]{Example}

\newtheorem{lemma}[theorem]{Lemma}
\newtheorem{convention}[theorem]{Convention}
\newtheorem{proposition}[theorem]{Proposition}
\newtheorem{remark}[theorem]{Remark}

\usepackage{listings}
\usepackage [all,arc,curve,color, arrow, matrix,frame]{xy}
\usepackage[english]{babel}
\usepackage{tikz}

\newcommand{\C}{{\mathbb C}}

\newcommand{\Real}{{\mbox{Re}}}
\newcommand{\supp}{{\mbox{supp}}}
\newcommand{\spann}{{\mbox{span}}}
\newcommand{\ran}{{\mbox{ran}}}

\newcommand{\idty}{{\mathbbm{1}}}

\numberwithin{equation}{section}
\numberwithin{theorem}{section} 
\numberwithin{footnote}{section}
\begin{document}

\title[Birman--{Kre\u\i n}--Vishik--Grubb theory for sectorial operators]{A Birman--{Kre\u\i n}--Vishik--Grubb theory for sectorial operators}

\author[C. Fischbacher]{Christoph Fischbacher$^1$}
\address{$^1$ Department of Mathematics, University of Alabama at Birmingham, Birmingham, AL 35294, USA}
\email{cfischb@uab.edu}

\numberwithin{equation}{section}
\numberwithin{theorem}{section} 
\numberwithin{footnote}{section}
\maketitle
%
\begin{abstract} We consider densely defined sectorial operators $A_\pm$ that can be written in the form $A_\pm=\pm iS+V$ with $\mathcal{D}(A_\pm)=\mathcal{D}(S)=\mathcal{D}(V)$, where both $S$ and $V\geq \varepsilon>0$ are assumed to be symmetric. We develop an analog to the Birmin-Kre\u\i n-Vishik-Grubb (BKVG) theory of selfadjoint extensions of a given strictly positive symmetric operator, where we will construct all maximally accretive extensions $A_D$ of $A_+$ with the property that $\overline{A_+}\subset A_D\subset A_-^*$. Here, $D$ is an auxiliary operator from $\ker(A_-^*)$ to $\ker(A_+^*)$ that parametrizes the different extensions $A_D$. After this, we will give a criterion for when the quadratic form $\psi\mapsto\Real\langle\psi,A_D\psi\rangle$ is closable and show that the selfadjoint operator $\widehat{V}$ that corresponds to the closure is an extension of $V$. We will show how $\widehat{V}$ depends on $D$, which --- using the classical BKVG-theory of selfadjoint extensions --- will allow us to define a partial order on the real parts of $A_D$ depending on $D$.
Applications to second order ordinary differential operators are discussed.
\end{abstract}
\section{Introduction}
In this paper, we want to study accretive extensions of sectorial operators $A_\pm$ of the form $A_\pm=\pm iS+V$, where $S$ and $V\geq \varepsilon>0$ \footnote{Here, for a symmetric operator $V$, the notation $V\geq\varepsilon$ means that $V$ is bounded from below with lower bound $\varepsilon$, i.e. for any $\psi\in\mathcal{D}(V)$, we have that $$\langle\psi,V\psi\rangle\geq\varepsilon\|\psi\|^2\:.$$} are both assumed to be symmetric but neither of them needs to be (essentially) selfadjoint.
\subsection{The Birman-Kre\u\i n-Vishik-Grubb theory of selfadjoint extensions} 
The study of abstract extension problems for operators on Hilbert spaces goes at least back to von Neumann \cite[Chapters V-VIII]{vNeumann}, whose well-known von Neumann formulae provide a full characterization of all selfadjoint extensions of a given closed symmetric operator $V$ with equal defect indices (for a presentation in a more modern terminology, see e.g. \cite[Vol. II, Sect.\ 80]{AkhiezerGlazman} or \cite[Satz 10.9]{Weid1}). In the same paper, von Neumann also discussed semibounded symmetric operators $V$ with lower bound $C>-\infty$, for which he managed to prove that for any $\varepsilon>0$, it is possible to construct a selfadjoint extension $V_{\varepsilon}$ of $V$ such that $V_{\varepsilon}$ is bounded from below by $(C-\varepsilon)$ (\cite[Satz 43]{vNeumann}). In particular, if $C>0$, this proves the existence of positive selfadjoint extensions of symmetric operators with positive semibound. The proof of this result relies on the construction of a non-negative selfadjoint extension $V_K$ of a given positive symmetric operator $V$, which is commonly known as the Kre\u\i n--von Neumann extension of $V$ (cf.\ \cite[Satz 42]{vNeumann}). In a footnote to the statement of \cite[Satz 43]{vNeumann}, he also conjectured the existence of a selfadjoint extension with the same lower bound as the initial symmetric operator.

This conjecture was answered in the affirmative by Friedrichs in \cite{Friedrichs}, who constructed what is nowadays known as the Friedrichs extension. Its construction exploits the fact that the quadratic form induced by a semibounded symmetric operator $V$ is always closable with its closure being the quadratic form associated to a selfadjoint extension $V_F$ that has the same lower bound as $V$. 

In \cite{Krein}, Kre\u\i n treated the problem of determining all non-negative selfadjoint extensions of a non-negative closed symmetric operator $V$ by considering the fractional linear transformation $F:=(V-\idty)(V+\idty)^{-1}$ on $\ran(V+\idty)$, whose compression $(P_{\text{ran}(V+\idty)}F)$ to $\ran(V+\idty)$ is selfadjoint ($P_{\text{ran}(V+\idty)}$ denotes the orthogonal projection onto $\ran(V+\idty)$). Moreover, if $V$ is non-negative, we have that $F$ is a contraction ($\|F\varphi\|\leq\|\varphi\|$ for all $\varphi\in\mathcal{D}(F)$, resp. $\|(V-\idty)f\|\leq\|(V+\idty)f\|$ for all $f \in\mathcal{D}(V)$). He showed that the problem of finding all non-negative selfadjoint extensions of $V$ is equivalent to finding all selfadjoint contractive extensions $F'$ of $F$ that are defined on the entire Hilbert space $\mathcal{H}$. Furthermore, he proved that there exist two special extensions of $V$, the above mentioned Kre\u\i n--von Neumann extension $V_K$ and the Friedrichs extension $V_F$. They are extremal in the sense that any other non-negative selfadjoint extension $\widehat{V}$ satisfies
\begin{equation*}
(V_F+\idty)^{-1}\leq (\widehat{V}+\idty)^{-1}\leq(V_K+\idty)^{-1}\:,
\end{equation*}
which is equivalent to
\begin{equation} \label{eq:fussball}
V_K\leq \widehat{V} \leq V_F
\end{equation}
in the quadratic form sense. Recall that for two non-negative selfadjoint operators $A$ and $B$ on a Hilbert space $\mathcal{H}$, the relation $A\leq B$ is defined as
\begin{equation*}
A\leq B :\Leftrightarrow \mathcal{D}(A^{1/2})\supset\mathcal{D}(B^{1/2})\:\:\text{and}\:\:\|A^{1/2}f\|\leq\|B^{1/2}f\|
\end{equation*}
for all $f\in\mathcal{D}(B^{1/2})$. \footnote{As done in \cite{Alonso-Simon}, we extend this definition to the case that $B$ is selfadjoint on a closed subspace $\mathcal{K}\subset\mathcal{H}$. For example, let $\mathcal{K}$ be a closed proper subspace of $\mathcal{H}$ and define $0_\mathcal{K}$ and $0_\mathcal{H}$ to be, respectively, the zero operators on $\mathcal{K}$ and $\mathcal{H}$. According to this definition, we then would get that $0_\mathcal{H}\leq 0_\mathcal{K}$. In \cite{Alonso-Simon}, the convention $B:=\infty$ on $\mathcal{D}(B)^\perp$ is introduced to make this more apparent.} 

  The further investigations of Vishik and Birman \cite{Vishik, Birman} resulted in the following characterization of all non-negative selfadjoint extensions of a positive closed symmetric operator $V$, which we present mainly following the notation and presentation of \cite{Alonso-Simon}:
\begin{proposition} Let $V\geq \varepsilon>0$ be a closed symmetric operator. Then, there is a one-to-one correspondence between all non-negative selfadjoint extensions of $V$ and all pairs $(\mathfrak{M},B)$, where $\mathfrak{M}\subset\ker(V^*)$ is a closed subspace and $B$ is a non-negative selfadjoint auxiliary operator in $\mathfrak{M}$ (in particular, $\overline{\mathcal{D}(B)}=\mathfrak{M}$). These non-negative selfadjoint extensions are given by
\begin{align}
V_{\mathfrak{M},B}:\:\:\mathcal{D}(V_{\mathfrak{M},B})&=\mathcal{D}(V)\dot{+}\{V_F^{-1}B\widetilde{k}+\widetilde{k}: \widetilde{k}\in\mathcal{D}(B)\}\dot{+}\{ V_F^{-1}k:k\in\ker(V^*)\cap\mathcal{D}(B)^\perp\}\notag\\ V_{\mathfrak{M},B}&=V^*\upharpoonright_{\mathcal{D}(V_{\mathfrak{M},B})}\:.
\label{eq:vishikbirmangrubb}
\end{align} \label{prop:kreinistfein}
\end{proposition} 
While these approaches predominantly relied on operator methods, the presentation of Alonso and Simon in \cite{Alonso-Simon} emphasizes form methods. They obtain the following description of the quadratic form induced by the operators $V_{\mathfrak{M},B}$:
\begin{align} \label{eq:alonsosimon}
V_{\mathfrak{M},B}^{1/2}:\qquad\qquad\mathcal{D}(V_{\mathfrak{M},B}^{1/2})&=\mathcal{D}(V_F^{1/2})\dot{+}\mathcal{D}(B^{1/2})\notag\\
\|V_{\mathfrak{M},B}^{1/2}(f+\eta)\|^2&=\|V_F^{1/2}f\|^2+\|B^{1/2}\eta\|^2\:,
\end{align}
where $f\in\mathcal{D}(V_F^{1/2})$ and $\eta\in\mathcal{D}(B^{1/2})$. From this, it immediately follows that
\begin{equation*}
V_{\mathfrak{M},B}\leq V_{\mathfrak{M}',B'}\quad\Leftrightarrow\quad B\leq B'
\end{equation*}
--- again with the understanding that $\mathfrak{M}'=\overline{\mathcal{D}(B')}$ is potentially a proper subspace of $\mathfrak{M}=\overline{\mathcal{D}(B )}$.
In particular, this implies \eqref{eq:fussball}, where $V_F=V_{\{0\},{\bf{0}}}$ and $V_K=V_{\ker(V^*), {\bf{0}}}$, where --- by abuse of notation --- ${\bf{0}}$ denotes the zero operator on the trivial space $\{0\}$ for $V_F$ as well as on $\ker(V^*)$ for $V_K$. 

These results have also been obtained and extended by Grubb in \cite[Chapter II \S 2]{Grubb68}, who was able to characterize (maximally) sectorial and (maximally) accretive extensions $\widehat{V}$ of $V$ such that $V\subset\widehat{V}\subset V^*$ by allowing the auxiliary operator $B$ to be (maximally) sectorial and (maximally) accretive (cf.\ also the addendum acknowledging Grubb's contributions to the field \cite{Grubbfriendly}). 
\subsection{Sectorial operators}
Numerous authors have contributed towards the theory of accretive extensions of given sectorial operators (these terms will be defined in Definitions \ref{def:accretive} and \ref{def:sectorial}). For a broader overview, which lies beyond the scope of this paper, we point the interested reader to the surveys \cite{Arli,AT2009} and all the references therein. However, let us mention the results of Arlinski\u\i , Derkach, Hassi, Kovalev, Malamud, Mogilevskii, Popov, de Snoo and {Tsekanovski\u\i} \cite{Arlinskii95, Kovalev, ArliPopov, ArliPopov2, AT2005, DM91, HMS04, MMM01, MMM06, MM97, MM99, MM02} who have made many contributions using form methods and boundary triples as well as dual pairs of contractions in order to determine maximally sectorial and maximally accretive extensions of a given sectorial operator.  

In particular, so called Vishik-Birman-Grubb type formulas \cite[Sec.\ 3.8]{Arli} that are a generalization of \eqref{eq:vishikbirmangrubb} to the sectorial setting have been obtained by {Arlinski\u\i} in \cite{Arlinskii99, Arlinskii2000}. (See also Theorems 3.21, 3.22 and 3.23 in \cite{Arli}.) Arlinski\u\i 's Vishik-Birman-Grubb type formulas are more general than our Vishik-Birman-Grubb type formula which we present in Theorem \ref{thm:dido} inasmuch as our assumptions are a special case of those made by Arlinski\u\i . However, we have chosen a different way of presenting our results, which resembles more to Proposition \ref{prop:kreinistfein}. Moreover, this allows us to build up towards our new results given in Sections \ref{sec:form} and \ref{sec:order}.

\subsection{Our results} In this paper, we want to consider sectorial operators $A_\pm$ that can be written in the form $A_\pm=\pm iS+V$ with $\mathcal{D}(A_\pm)=\mathcal{D}(S)=\mathcal{D}(V)$, where $S$ and $V\geq \varepsilon>0$ are assumed to be symmetric. To this end, we will introduce a slightly more general notion of sectoriality than can be usually found in the literature (cf. for example \cite[Chap. V,\S 3]{Kato}). In particular, for an accretive operator to be sectorial, we will only require that its numerical range be contained in some sector of the right half-plane with opening angle strictly less than $\pi$, while one usually has the additional condition that this sector be strictly contained inside the right half-plane, i.e.\ it may not contain any purely imaginary number.   \\
We will proceed as follows:\\
In Section \ref{sec:definitions}, we will give a few definitions and review some previous results. We will also discuss the Friedrichs extension $A_F$ of a given sectorial operator $A_+=iS+V$ and show some useful properties (Section \ref{subsec:Friedrichs}).   \\
In Section \ref{sec:proper}, in analogy to the classical Birman-Kre\u\i n-Vishik-Grubb (BKVG) theory, we use an auxiliary operator $D$ mapping from a subspace $\mathfrak{N}\subset\ker(A_-^*)\cap\mathcal{D}(V_K^{1/2})$ into $\ker(A_+^*)$ in order to characterize all maximally accretive extensions $A_{\mathfrak{N},D}$ of $\overline{A_+}$ that satisfy $\overline{A_+}\subset A_{\mathfrak{N},D}\subset A_-^*$ (Theorem \ref{thm:dido}). Here, $\overline{A_+}$ denotes the closure of $A_+$. In this sense, one can think of $\overline{A_+}$ as a minimal operator and of $A_-^*$ as a maximal operator, so that $A_{\mathfrak{N},D}$ would describe an accretive extension of $\overline{A_+}$ that still preserves the action of $A_-^*$ (cf. Example \ref{ex:minmax}). Even though the results discussed in this section are a special case of results previously obtained in \cite{FNW}, they are presented in a way that directly shows how they generalize the traditional BKVG-theory for the non-negative symmetric case and that allows us to obtain the results in the following sections. 

In Section \ref{sec:form} we then investigate the question whether the quadratic form $\mathfrak{re}_{D,0}$, which is associated to the real part of $A_{\mathfrak{N},D}$ and is given by
\begin{align*}
\mathfrak{re}_{D,0}:\qquad\mathcal{D}(\mathfrak{re}_{D,0})&=\mathcal{D}(A_{\mathfrak{N},D})\\
\psi&\mapsto\Real\langle\psi,A_{\mathfrak{N},D}\psi\rangle
\end{align*}
is closable (Theorem \ref{thm:closable}). If yes, we construct its closure $\mathfrak{re}_D$ (Lemma \ref{lemma:sts} and Theorem \ref{thm:townsende}). We will also see that $\mathfrak{re}_{D,0}$ can only fail to be closable in the extremal case
\begin{equation} \label{eq:extremal}
\inf\{\arg(z):z\in\mathcal{N}(A_+)\}=-\pi/2\quad\text{or}\quad
\sup\{\arg(z):z\in\mathcal{N}(A_+)\}=\pi/2\:,
\end{equation}
where $\mathcal{N}(A_+)$ denotes the numerical range of the operator $A_+$ (Example \ref{ex:closfail}).

We then show in Section \ref{sec:order} that the selfadjoint operator ${V}_D$ associated to $\mathfrak{re}_D$ is an extension of $V$ (Lemma \ref{lemma:brexit}), which implies by the classical BKVG theory (Proposition \ref{prop:kreinistfein}) of selfadjoint extensions that there exists a selfadjoint operator $B$ defined on a subspace $\mathfrak{M}$ of $\ker V^*$ such that ${V}_D=V_{\mathfrak{M},B}$. We will also show how $\mathfrak{M}$ and $B$ can be constructed from $\mathfrak{N}$ and $D$ and how this can be used to define an order on the real parts of different extensions $A_{\mathfrak{N}_1,D_1}$ and $A_{\mathfrak{N}_2,D_2}$ of $A$. In particular, we will obtain a generalization of Proposition \ref{prop:kreinistfein}, which we will recover for the case that $S=0$, i.e.\ if $A_\pm=V$.
\section{Definitions and previous results}  
\label{sec:definitions}
\subsection{Definitions}
Let us start with a few basic definitions:
\subsubsection{Accretive and sectorial operators}
\begin{definition} A densely defined operator $A$ on a Hilbert space $\mathcal{H}$ is called {\bf{accretive}} if and only if 
\begin{equation*}
\Real\langle\psi,A\psi\rangle\geq 0
\end{equation*}
for any $\psi\in\mathcal{D}(A)$. Moreover, if $A$ is accretive and has no non-trivial accretive operator extension then it is called {\bf{maximally accretive}}.
\label{def:accretive} 
\end{definition}
\begin{remark} Note that we have defined the inner product $\langle \cdot,\cdot\rangle$ to be antilinear in the first and linear in the second component, i.e.\ for any $f,g\in\mathcal{H}$ and any $\lambda\in\C$ we get $\langle f,\lambda g\rangle=\lambda\langle f,g\rangle=\langle \overline{\lambda}f,g\rangle$.
\end{remark}
\begin{definition} Let $A$ be accretive. We say that $A$ is {\bf{sectorial}} if there exists a $\varphi\in (0,2\pi)$ such that $(e^{i\varphi}A)$ is accretive as well. Moreover, if $A$ is sectorial and maximally accretive then it is called {\bf{maximally sectorial}}.\label{def:sectorial}
\end{definition}
\begin{remark}
Again, we emphasize that in the literature (e.\ g.\ in \cite[Chap. V,\S 3]{Kato}), an accretive operator $A$ is usually called sectorial if there exists an $\varepsilon\in (0,2\pi)$ such that $(e^{i\varepsilon}A)$ and $(e^{-i\varepsilon}A)$ are both still accretive. Our slightly more general definition will however allow us to consider interesting examples where the quadratic form induced by the real part of a maximally sectorial extension is not closable (Example \ref{ex:closfail}), which would otherwise not be possible (cf.\ Remark \ref{rem:closable}).
\end{remark}
\begin{example} Let $\mathcal{H}=L^2(0,1)$ and for $\gamma>0$, define the operators $C_\pm$ as follows:
\begin{align*}
C_\pm:\qquad \mathcal{D}(C_\pm)&=\mathcal{C}_c^\infty(0,1)\\
\left(C_\pm f\right)(x)&=\pm if''(x)+\frac{\gamma}{x^2}f(x)\:.
\end{align*}
It can be immediately seen that $C_\pm$ are both accretive and sectorial in the sense of Definition \ref{def:sectorial}, since
\begin{equation*}
\Real\langle f,C_\pm f\rangle=\gamma\int_0^1\frac{|f(x)|^2}{x^2}dx\geq 0\quad\text{and}\quad \Real\langle f, \pm iC_\pm f\rangle=\int_0^1|f'(x)|^2dx\geq 0
\end{equation*}
for any $f\in\mathcal{C}_c^\infty(0,1)$.
 Moreover, it can be shown that the operators $C_\pm$ are not sectorial in the sense of \cite[Chap. V,\S 3]{Kato}. To this end, let $0<\varepsilon<\pi$ and consider $e^{\pm i\varepsilon}C_\pm$. For any $f\in\mathcal{C}_c^\infty(0,1)$ we then get
\begin{align} \label{eq:nokato}
\Real\langle f,\left(e^{\pm i\varepsilon}C_\pm\right)f\rangle=\int_0^1\left( \cos(\varepsilon)\frac{\gamma|f(x)|^2}{x^2}-\sin(\varepsilon)|f'(x)|^2\right)dx
\end{align}
Now, choose a sequence of normalized compactly supported smooth functions $\{f_n\}_{n=1}^\infty$ with $\supp(f_n)\subset[1/2,1]$ such that $\|f_n'\|\rightarrow\infty$. From this, it immediately follows that
\begin{equation*}
\lim_{n\rightarrow\infty}\Real\langle f_n,\left(e^{\pm i\varepsilon}C_\pm\right)f_n\rangle=-\infty\:,
\end{equation*}
which shows that the operators $C_\pm$ are not sectorial in the sense of \cite[Chap. V,\S 3]{Kato}. \label{ex:closfailprep}
\end{example}
Finally, let us introduce some useful terminology for sectorial operators of the form $A_\pm=\pm iS+V$, where $S$ and $V\geq\varepsilon$ are symmetric:
\begin{definition}
Let $S$ and $V\geq \varepsilon$ be symmetric such that $\mathcal{D}(S)=\mathcal{D}(V)$ and such that the operators $A_\pm:=\pm iS+V$ are sectorial. We then call $V$ the {\bf{real part}} of $A_\pm$ and $(\pm S)$ the {\bf{imaginary part}} of $A_\pm$. 
\end{definition} 
\subsubsection{Proper extensions of dual pairs}
Since we will construct extensions $\widehat{A}$ of $A_+$ with the property that $\overline{{A}_+}\subset\widehat{A}\subset A_-^*$, let us introduce the notion of dual pairs (cf.\ \cite{Edmunds-Evans, LyantzeStorozh} for details): 
\begin{definition} Let $(A,\widetilde{A})$ be a pair of densely defined and closable operators. We say that they form a {\bf{dual pair}} if 
\begin{equation*}
A\subset \widetilde{A}^*\quad\text{resp.}\quad \widetilde{A}\subset A^*\:.
\end{equation*}
In this case, $A$ is called a {\bf{formal adjoint}} of $\widetilde{A}$ and vice versa. Moreover, an operator $\widehat{A}$ such that $A\subset\widehat{A}\subset\widetilde{A}^*$ is called a {\bf{proper extension}} of the dual pair $(A,\widetilde{A})$.
\label{def:dual}
\end{definition}   
\begin{remark} Note that by \cite[Chap.\ V, Thm. 3.4]{Kato}, an accretive operator is always closable.
\end{remark}
\begin{remark} Let $A$ be an accretive and thus closable operator with closure $\overline{A}$. In addition, let $B$ be a closed accretive extension of $A$, which means that $\overline{A}\subset B$. Thus, when considering the problem of finding closed accretive extensions of $A$, it is sufficient to restrict the search to closed extensions of $\overline{A}$, which we will do from now on.
\end{remark}
For the framework of accretive operators of the form $A_\pm=\pm iS+V$, where $\mathcal{D}(A_\pm)=\mathcal{D}(S)=\mathcal{D}(V)$, it is not hard to see that $(A_+,A_-)$ as well as $(\overline{A}_+,\overline{A}_-)$ form a dual pair. In particular, since $\overline{A_+}\subset A_-^*$, one can think of $\overline{A_+}$ as a minimal operator and of $A_-^*$ as the corresponding maximal operator. Then, the problem of finding proper extensions $\widehat{A}$ such that $\overline{A_+}\subset\widehat{A}\subset A_-^*$ can be interpreted as finding extensions of $\overline{A_+}$ that still preserve the formal action of the maximal operator $A_-^*$:
\begin{example} \label{ex:minmax}
Let $\mathcal{H}=L^2(0,1)$ and let $S=-\frac{d^2}{dx^2}$ be the Laplace operator with domain $\mathcal{D}(S)=\mathcal{C}_c^\infty(0,1)$ and let $V$ be a multiplication by a real function $V(x)\in L^\infty(0,1)$ satisfying $V(x)\geq\varepsilon>0$ almost everywhere. Moreover, let $\mathcal{D}(V)=\mathcal{D}(S)=\mathcal{C}_c^\infty(0,1)$. Then, with the choice $A_\pm:=\pm iS+V$, it is clear that $(A_+,A_-)$ is a dual pair of sectorial operators. Moreover, the maximal operator $A_-^*$ is given by
\begin{align*}
A_-^*:\qquad\mathcal{D}(A_-^*)&=H^2(0,1)\\
\left(A_-^*f\right)(x)&=-if''(x)+V(x)f(x)\:,
\end{align*}
where $f''(x)$ denotes the second weak derivative of $f$, while the minimal operator $\overline{A_+}$ is the restriction of $A_-^*$ to $H_0^2(0,1)$. \end{example}
This means that the extension problem of finding suitable $\widehat{A}$ such that $\overline{A_+}\subset\widehat{A}\subset A^*_-$ reduces to finding suitable domains $\mathcal{D}(\overline{A_+})\subset\mathcal{D}(\widehat{A})\subset\mathcal{D}(A_-^*)$ to which we restrict $A_-^*$, i.e. $\widehat{A}=A_-^*\upharpoonright_{\mathcal{D}(\widehat{A})}$. Let us thus introduce the following useful notation for complementary subspaces:
\begin{definition} Let $\mathcal{N}$ be a (not necessarily closed) linear space and $\mathcal{M}\subset\mathcal{N}$ be a (not necessarily closed) subspace. With the notation $\mathcal{N}//\mathcal{M}$ we mean any subspace of $\mathcal{N}$, which is complementary to $\mathcal{M}$, i.e.
\begin{equation*}
(\mathcal{N}//\mathcal{M})+\mathcal{M}=\mathcal{N}\quad\text{and}\quad(\mathcal{N}//\mathcal{M})\cap\mathcal{M}=\{0\}\:.
\end{equation*}
\end{definition}
This means that we can use subspaces $\mathcal{V}\subset\mathcal{D}(A_-^*)//\mathcal{D}(\overline{A}_+)$ in order to parametrize all possible proper extensions $\widehat{A}$ such that $\overline{A}_+\subset\widehat{A}\subset A_-^*$:
\begin{definition} Let $\mathcal{V}\subset\mathcal{D}(A_-^*)//\mathcal{D}(\overline{A_+})$ be a subspace. Then, the operator $A_{+,\mathcal{V}}$ is defined as
\begin{align*}
A_{+,\mathcal{V}}:\qquad\mathcal{D}(A_{+,\mathcal{V}})=\mathcal{D}(\overline{A_+})\dot{+}\mathcal{V}, \quad
A_{+,\mathcal{V}}&=A_-^*\upharpoonright_{\mathcal{D}(A_{+,\mathcal{V}})}\:.
\end{align*} \label{def:subspaceextension}
\end{definition}
Given a dual pair of minimal operators $(A,\widetilde{A})$, the following result provides a useful description for the domains of the maximal operators $\widetilde{A}^*$ and $A^*$:
\begin{proposition}[{\cite[Chapter  II, Lemma 1.1]{Grubb68}}] \label{thm:EEvM}
Let $(A,\widetilde{A})$ be a dual pair with $\lambda\in\widehat{\rho}(A)$ and $\overline{\lambda}\in\widehat{\rho}(\widetilde{A})$, where $\widehat{\rho}(C)$ denotes the field of regularity of an operator $C$\footnote{Recall that the field of regularity of an operator $C$ is given by
\begin{equation*}
\widehat{\rho}(C)=\{z\in\C: \exists c(z)>0\:\:\text{such that}\:\:\forall\: f\in\mathcal{D}(C):\:\: \|(C-z)f\|\geq c(z)\|f\|\}\:.
\end{equation*}  }. Then there exists a proper extension $\widehat{A}$ of $(A,\widetilde{A})$ such that $\lambda\in\rho(\widehat{A})$ and $\mathcal{D}(\widetilde{A}^*)$ can be expressed as
\begin{equation} \label{eq:malamud}
\mathcal{D}(\widetilde{A}^*)=\mathcal{D}(\overline{A})\dot{+}(\widehat{A}-\lambda)^{-1}\ker(A^*-\overline{\lambda})\dot{+}\ker(\widetilde{A}^*-{\lambda})\:.
\end{equation}
Likewise, we get the following description for $\mathcal{D}(A^*)$:
\begin{equation*}
\mathcal{D}({A}^*)=\mathcal{D}(\overline{\widetilde{A}})\dot{+}(\widehat{A}^*-\overline{\lambda})^{-1}\ker(\widetilde{A}^*-\lambda)\dot{+}\ker({A}^*-\overline{{\lambda}})\:.
\end{equation*}
\end{proposition}
The following proposition ensures the existence of at least one proper maximally accretive extension of any dual pair of accretive operators:
\begin{proposition}[{\cite[Chapter IV, Proposition 4.2]{NagyFoias}}] \label{thm:friendly}
Let $(A,\widetilde{A})$ be a dual pair of accretive operators. Then there exists a maximally accretive proper extension of the dual pair $(A,\widetilde{A})$.
\end{proposition}
\subsection{Previous results on proper accretive extensions of dual pairs}
For our later results we will need the following proposition which gives a necessary and sufficient condition for a proper extension $A_{+,\mathcal{V}}$ of a dual pair $(A_+,A_-)$ of accretive operators to be accretive itself:
\begin{proposition} \label{prop:accretivetelemann} Let $S$ and $V\geq 0$ be symmetric operators such that $\mathcal{D}(S)=\mathcal{D}(V)$ and let $V_K$ be the selfadjoint non-negative {Kre\u\i n}-von Neumann extension of $V$. Moreover, let $A_\pm:=(\pm iS+V)$ and let $\mathcal{V}\subset\mathcal{D}(A_-^*)//\mathcal{D}(\overline{A_+})$ be a linear space. Then, $A_{+,\mathcal{V}}$ is a proper accretive extension of $(\overline{A}_+,\overline{A}_-)$ if and only if $\mathcal{V}\subset\mathcal{D}(V_K^{1/2})$ and for any $v\in\mathcal{V}$ it holds that
\begin{equation} \label{eq:simmern}
q(v):=\Real\langle v,A_-^*v\rangle-\|V_K^{1/2}v\|^2\geq 0\:.
\end{equation}
\end{proposition}
\begin{proof}
In \cite[Lemma\ 3.3]{Nonproper}, it was shown that $(\overline{A}_+,\overline{A}_-)$ is --- up to a suitable multiplication by $\pm i$ --- a dual pair satisfying the assumptions of \cite[Thm.\ 4.7]{FNW} from which the proposition immediately follows.
\end{proof}
Finally, we will also need to make use of the following result which --- again up to a suitable multiplication by $\pm i$ --- can be found in {\cite[Lemma 5.1]{FNW}}.
\begin{lemma} Let $(A_+,A_-)$ be a dual pair satisfying the assumptions of Proposition \ref{prop:accretivetelemann} and let $A_{+,\mathcal{V}}$ be a proper accretive extension. We then get that
\begin{equation} \label{eq:leberkaas}
\Real\langle f+v,A_{+,\mathcal{V}}(f+v)\rangle=\|V_K^{1/2}(f+v)\|^2+q(v)
\end{equation}
for any $f\in\mathcal{D}(\overline{A_+})$ and any $v\in\mathcal{V}$, where $q(\cdot)$ is the quadratic form defined in \eqref{eq:simmern}.
\label{lemma:cernohorsky}
\end{lemma}
\subsection{The Friedrichs extension of a sectorial operator} \label{subsec:Friedrichs}
The Friedrichs extension of a sectorial operator will play an important role for the following results. Since it is mainly a form construction, let us recall a few important definitions.
We begin with the definition of a closable quadratic form:
\begin{definition}[{Closable quadratic form, cf.\ \cite[VI, \S 1, Sec.\ 4]{Kato}}]
Let $q$ be a quadratic form. Then, $q$ is called closable if and only if for any sequence $\{f_n\}_{n=1}^\infty\subset\mathcal{D}(q)$, we have that if
$$ \|f_n\|\overset{n\rightarrow\infty}{\longrightarrow}0\quad\text{and}\quad q(f_n-f_m)\overset{n,m\rightarrow\infty}{\longrightarrow}0\:,$$
then this implies that $$q(f_n)\overset{n\rightarrow\infty}{\longrightarrow}0\:.$$ \label{def:closable}
\end{definition}
\begin{remark} \normalfont If $q$ is closable, its closure $q'$ is given by \cite[VI, Thm.\ 1.17]{Kato}
\begin{align*}
q':\quad\mathcal{D}(q')&=\{ f\in\mathcal{H}: \exists \{f_n\}_{n=1}^\infty\subset\mathcal{D}(q)\:\: s.t.\:\: \|f_n-f\|\overset{n\rightarrow\infty}{\longrightarrow}0\:\:\text{and}\:\: q(f_n-f_m)\overset{n,m\rightarrow\infty}{\longrightarrow} 0\}\\
q'(f)&=\lim_{n\rightarrow\infty}q(f_n)\:.
\end{align*}
\end{remark}
For a sectorial operator $A$, we can define its Friedrichs extension $A_F$. In the literature (e.g. in \cite{Kato}), this is usually done for sectorial operators with angle $\eta$, i.e. for operators which have numerical range contained in the set $\{z\in\C:-\eta\leq\arg(z)\leq \eta\}$ for some $0\leq\eta<\frac{\pi}{2}$, but we will give a proof for our more general notion of sectorial operators:
 
\begin{proposition} \label{thm:katocool}
Let $T$ be sectorial and let $s_T$ be the quadratic form induced by $T$, i.e.\
\begin{align*}
s_T:\quad\mathcal{D}(s_T)&=\mathcal{D}(T)\\
\psi&\mapsto s_T(\psi):=\langle \psi,T\psi\rangle\:.
\end{align*}
Then, $s_T$ is closable, where we denote its closure by $s_{T_F}$. The form domain $\mathcal{Q}(T)$ of $s_{T_F}$ is $\mathcal{Q}(T):=\overline{\mathcal{D}(T)}^{\|\cdot\|_T}$, with the norm $\|\cdot\|_T$ being given by 
\begin{equation} \label{eq:realnorm}
\|\psi\|_{T}^2:=\|\psi\|^2+\Real \langle \psi,e^{i\varphi}T\psi\rangle\:,
\end{equation}
where $e^{i\varphi}$ is any complex phase such that $\left(e^{i(\varphi\pm\varepsilon)}T\right)$ is still accretive for a sufficiently small $\varepsilon>0$. The form domain $\mathcal{Q}(T)$ does not depend on the specific choice of $\varphi$.
The Friedrichs extension of $T$ --- denoted by $T_F$ --- is the operator associated to $s_{T_F}$, i.e. it is given by
\begin{align*}
T_F:\quad\mathcal{D}(T_F)&=\{f\in\mathcal{Q}(T):\exists w\in\mathcal{H}\:\: s.t.\:\:\forall g\in\mathcal{Q}(T):\:\: s_{T_F}(f,g)=\langle w,g\rangle\}\\
f&\mapsto w\:.
\end{align*}
Here, $s_{T_F}(\cdot,\cdot)$ denotes the sesquilinear form associated to $s_{T_F}$ that can be obtained by polarization.

 The operator $T_F$ is maximally sectorial and the closures of the numerical ranges of $T$ and $T_F$ coincide.

 Moreover, we have the following description of $T_F^*$:
 \begin{align} \label{eq:adjointstar} 
 T_F^*:\quad\mathcal{D}(T_F^*)&=\mathcal{Q}(T)\cap\mathcal{D}(T^*)\notag\\
 T_F^*&=T^*\upharpoonright_{\mathcal{D}(T_F^*)}\:.
 \end{align}
\end{proposition}
\begin{proof}
For the construction of the maximally sectorial Friedrichs extension, we refer to \cite[VI, Theorem 1.27, Theorem 2.1, Corollary 2.4 and VI, \S 2.3]{Kato}. There is only a slight subtlety for the extremal case \eqref{eq:extremal}, in which case one has to choose a nonzero $\varphi$ in \eqref{eq:realnorm}. It is not hard to see that $\mathcal{Q}(T)$ does not depend on the specific choice of $\varphi$ as long as $\left(e^{i(\varphi\pm\varepsilon)}T\right)$ is still accretive for a sufficiently small $\varepsilon>0$, since the norms 
 $\psi\mapsto\|\psi\|^2+\Real \langle \psi,e^{i\varphi_1}T\psi\rangle$ and $\psi\mapsto\|\psi\|^2+\Real \langle \psi,e^{i\varphi_2}T\psi\rangle$ are equivalent for any such $\varphi_1$ and $\varphi_2$.\\
The assertion about the closures of the numerical ranges of $T$ and $T_F$ coinciding follows from \cite[VI, Theorem 1.18 and Corollary 2.3]{Kato}.\\
For \eqref{eq:adjointstar}, cf.\ \cite[Remarks right after Thm.\ 1]{Arlinskii97}.
\end{proof}
Now, consider dual pairs $(A_+,A_-)$ as described in Proposition \ref{prop:accretivetelemann}. Let $A_{\pm,F}$ denote the Friedrichs extension of $A_\pm$ respectively. Let us now show that $A_{+,F}=(A_{-,F})^*$:
\begin{theorem} \label{prop:friedrichsfreude}
Let $(A_+,A_-)$ be dual pair of sectorial operators of the form $A_\pm=\pm iS+V$, where $S$ and $V$ are as in Proposition \ref{prop:accretivetelemann}. We then have $A_{+,F}=(A_{-,F})^*$, which implies in particular that $A_{+,F}$ is a proper maximally sectorial extension of the dual pair $(A_+,A_-)$.
\end{theorem}
\begin{proof}
Firstly, observe that if in the extremal case \eqref{eq:extremal}, we can always consider sectorial dual pairs of the form $(e^{i\varphi} A_+, e^{-i\varphi} A_-)$, where $\varphi$ is chosen such that we are not in the extremal case \eqref{eq:extremal}. If we show that $(e^{i\varphi} A_+)_F=(e^{-i\varphi} A_-)_F^*$, this readily implies that $A_{+,F}=(A_{-,F})^*$. It is thus sufficient to only consider the non-extremal case. This means that for the norms induced by $A_\pm$ as in Equation \eqref{eq:realnorm}, we can always make the choice $\varphi= 0$, i.e.
\begin{equation*}
\|\psi\|_{A_\pm}^2=\|\psi\|^2+\Real\langle\psi,A_{\pm}\psi\rangle\:.
\end{equation*}
These two norms are easily shown to be equal for any $\psi\in\mathcal{D}(A_+)=\mathcal{D}(A_-)=\mathcal{D}(S)=\mathcal{D}(V)$ since
\begin{equation} \label{eq:mindestlohn}
\Real\langle \psi,A_+\psi\rangle=\Real\langle \psi,(iS+V)\psi\rangle=\langle \psi,V\psi\rangle=\Real\langle \psi,(-iS+V)\psi\rangle=\Real\langle \psi,A_-\psi\rangle\:.
\end{equation}
Now, let $s_{A_{\pm}}$ denote the quadratic quadratic forms given by:
\begin{align*}
s_{A_{\pm}}:\qquad\mathcal{D}(s_{A_{\pm}})&=\mathcal{D}(A_\pm)\\
\psi&\mapsto\langle \psi,A_\pm \psi\rangle
\end{align*}
and let $s_{A_{\pm}}(\cdot,\cdot)$ denote the associated sesquilinear form, which is obtained from polarization.
Moreover, let $s_{A_{\pm,F}}$ denote the respective closures of the quadratic forms $s_{A_\pm}$.
Now, Equation \eqref{eq:mindestlohn} implies that the form domains as described in Proposition \ref{thm:katocool} are equal, since $$\mathcal{Q}(A_+)=\overline{\mathcal{D}(A_+)}^{\|\cdot\|_{A_+}}=\overline{\mathcal{D}(A_-)}^{\|\cdot\|_{A_-}}=\mathcal{Q}(A_-)\:.$$
Thus, for any $\phi,\psi\in\mathcal{Q}(A_+)=\mathcal{Q}(A_-)$, there exist sequences $\{\phi_n\}_{n=1}^\infty,\{\psi_n\}_{n=1}^\infty\subset\mathcal{D}(A_+)=\mathcal{D}(A_-)$ such that 
\begin{align*}
s_{A_{+,F}}(\phi,\psi)&=\lim_{n\rightarrow\infty}s_{A_{+,F}}(\phi_n,\psi_n)=\lim_{n\rightarrow\infty}\langle\phi_n,A_+\psi_n\rangle=\lim_{n\rightarrow\infty}\langle\phi_n,(iS+V)\psi_n\rangle=\overline{\lim_{n\rightarrow\infty}\langle\psi_n,(-iS+V)\phi_n\rangle}\\&=\overline{\lim_{n\rightarrow\infty}\langle\psi_n,A_-\phi_n\rangle}=\overline{s_{A_{-,F}}(\psi,\phi)}=s_{(A_{-,F})^*}(\phi,\psi)\:,
\end{align*}
where the last equality follows from \cite[VI, Thm.\ 2.5]{Kato}. This implies that $A_{+,F}=(A_{-,F})^*$. Now, since $A_\pm\subset A_{\pm,F}$, this immediately yields
\begin{equation*}
A_+\subset A_{+,F}=(A_{-,F})^*\subset A_-^*\:,
\end{equation*}
showing that $A_{+,F}$ is a proper extension of $(A_+,A_-)$ and thus finishing the proof.
\end{proof} 
\begin{remark} We like to point out that for a generic dual pair $(A,\widetilde{A})$ of sectorial operators, it is not necessarily true that $A_F^*=\widetilde{A}_F$. For instance, let $\mathcal{H}=L^2(0,1)$ and consider the two symmetric operators $S$ and $\widetilde{S}$ given by
\begin{align*}
S:\qquad\mathcal{D}(S)&=\{f\in H^2(0,1): f(0)=f'(0)=f(1)=f'(1)=0\},\quad f\mapsto -f''\\
\widetilde{S}:\qquad\mathcal{D}(\widetilde{S})&=\{f\in H^2(0,1): f(0)=f'(0)=f'(1)=0\}, \qquad\qquad f\mapsto -f''
\end{align*}
and let $V\geq 0$ be an arbitrary non-negative bounded operator. Defining $A:=S+iV$ and $\widetilde{A}=\widetilde{S}-iV$, it is not hard to see that $(A,\widetilde{A})$ is a dual pair. However, since $S_F$ is the Laplacian on $(0,1)$ with Dirichlet boundary conditions at both endpoints of the interval while $\widetilde{S}_F$ is the Laplacian with a Dirichlet condition at $0$ and a Neumann condition at $1$, this implies that $S_F\neq \widetilde{S}_F$ and consequently $A_F^*\neq \widetilde{A}_F$.
\end{remark}
Now, let us combine the results of Theorem \ref{prop:friedrichsfreude} and Proposition \ref{thm:EEvM} in order to find a convenient description of $\mathcal{D}(A_-^*)$:
\begin{corollary} Let $A_\pm=\pm iS+V$, where $S$ and $V\geq\varepsilon>0$ are both symmetric and such that $A_\pm$ are sectorial. Moreover, assume that their domains satisfy $\mathcal{D}(S)=\mathcal{D}(V)$. Then the domain of $A_-^*$ is given by
\begin{equation} \label{eq:domainformula}
\mathcal{D}(A_-^*)=\mathcal{D}(\overline{A_+})\dot{+}A_{+,F}^{-1}\ker(A_+^*)\dot{+}\ker(A_-^*)\:.
\end{equation} 
\label{coro:domain}
\end{corollary} 
\begin{proof} Since $V\geq \varepsilon>0$, it follows that $0\in\widehat{\rho}(A_+)\cap\widehat{\rho}(A_-)$, which by Proposition \ref{thm:EEvM} implies that there exists a proper extension $\widehat{A}$ of $(A_+,A_-)$ with $0\in\rho(\widehat{A})$ such that
\begin{equation*}
\mathcal{D}(A_-^*)=\mathcal{D}(\overline{A_+})\dot{+}\widehat{A}^{-1}\ker(A_+^*)\dot{+}\ker(A_-^*)\:.
\end{equation*}
On the other hand, $V\geq\varepsilon>0$ also implies that $0\in\rho(A_{+,F})$, from which we get:
\begin{equation*} 
\mathcal{D}(A_{+,F})=A_{+,F}^{-1}\mathcal{H}=A_{+,F}^{-1}\left(\overline{\ran(A_+)}\oplus\ker(A_+^*)\right)=A_{+,F}^{-1}\ran(\overline{A_+})\dot{+}A_{+,F}^{-1}\ker(A_+^*)=\mathcal{D}(\overline{A_+})\dot{+}A_{+,F}^{-1}\ker(A_+^*)\:,
\end{equation*}
where $\overline{\ran(A_+)}=\ran(\overline{A_+})$ follows from the fact that $0\in\widehat{\rho}(A_+)$.
From Theorem \ref{prop:friedrichsfreude}, we now get that $A_{+,F}$ is a proper extension of $(A_+,A_-)$, which implies in particular that $\mathcal{D}(A_{+,F})=\mathcal{D}(\overline{A_+})\dot{+}A_{+,F}^{-1}\ker(A_+^*)\subset\mathcal{D}(A_-^*)$. Moreover, since it trivially holds that $\ker(A_-^*)\subset\mathcal{D}(A_-^*)$, we get the $``\supset"$ inclusion in \eqref{eq:domainformula}.
Let us now show the $``\subset"$ inclusion in \eqref{eq:domainformula}. To this end, let $\widehat{k}_+\in \ker(A_+^*)$ be arbitrary and consider \[\eta:=\widehat{A}^{-1}\widehat{k}_+-A_{+,F}^{-1}\widehat{k}_+\:,\] which is an element of $\ker(A_-^*)$. Hence, for any $f_+\in\mathcal{D}(\overline{A}_+)$, $\widehat{k}_+\in\ker(A_+^*)$ and $\widehat{k}_-\in\ker(A_-^*)$, we get that
\[ \mathcal{D}(A_-^*)\ni f_++\widehat{A}^{-1}\widehat{k}_++\widehat{k}_-=f_++A_{+,F}^{-1}\widehat{k}_++(\eta+\widehat{k}_-)\in\mathcal{D}(\overline{A_+})\dot{+}A_{+,F}^{-1}\ker(A_+^*)\dot{+}\ker(A_-^*)\:, \]
which shows the corollary.
\end{proof}
\begin{example} Let $\gamma> 0$ and consider the dual pair of operators $(A_+,A_-)$ given by
\begin{align*}
A_\pm:\mathcal{D}(A_\pm)=\mathcal{C}_c^\infty(0,1)\\
(A_\pm f)(x)= \pm i\frac{\gamma}{x^2}f(x)-f''(x)\:.
\end{align*}\
Using Hardy's inequality, it can be shown that the numerical range of $A_+$ is contained in a sector that lies strictly within the right half-plane of $\C$, which means that we are not in the extremal case \eqref{eq:extremal}. We therefore can use the norm induced by the real part of $A_\pm$ in order to construct $\mathcal{Q}(A_\pm)$. Thus, for any $f\in\mathcal{C}_c^\infty(0,1)$, we get
\begin{equation*}
\|f\|_{A_\pm}^2=\|f\|^2+\Real\langle f,A_\pm f\rangle=\|f\|^2+\|f'\|^2=\|f\|_{H^1}^2
\end{equation*}
i.e.\ the first Sobolev norm. Moreover, the form domains are given by
\begin{equation*}
\mathcal{Q}(A_\pm)=\overline{\mathcal{C}_c^\infty(0,1)}^{\|\cdot\|_{H^1}}=H^1_0(0,1)\:.
\end{equation*} 
Now, let us define the numbers $\omega_\pm:=\frac{1\pm\sqrt{1+4i\gamma}}{2}$. Depending on $\gamma$, we then have the following two situations: 
\begin{itemize}
\item $\gamma\geq\sqrt{3}$: In this case, we get
\begin{equation*}
\ker(A_-^*)=\spann\{x^{\omega_+}\}\quad\text{and}\quad\ker(A_+^*)=\spann\{x^{\overline{\omega_+
}}\}\:,
\end{equation*}
since $\Real(\omega_-)\leq-1/2$ for $\gamma\geq\sqrt{3}$. In this case, it can be shown that
\begin{equation} \label{eq:domainbigg}
\mathcal{D}(A_-^*)=\mathcal{D}(\overline{A_+})\dot{+}\spann\{x^{\overline{\omega_+
+2}}\}\dot{+}\spann\{x^{\omega_+}\}\;.
\end{equation}
Hence, using Equation \eqref{eq:adjointstar} from Theorem \ref{thm:katocool} and the equality $A_{+,F}=(A_{-,F})^*$, which was shown in Theorem \ref{prop:friedrichsfreude}, we find that the domain of $A_{+,F}$ is given by
\begin{equation*}
\mathcal{D}(A_{+,F})=\mathcal{D}((A_{-,F})^*)=\mathcal{Q}(A_-)\cap\mathcal{D}(A_-^*)=H^1_0(0,1)\cap\mathcal{D}(A_{-}^*)=\mathcal{D}(\overline{A_+})\dot{+}\spann\{x^{\overline{\omega_++2}}-x^{\omega_+}\}\:.
\end{equation*}
\item $0<\gamma<\sqrt{3}$: In this case, we get
\begin{equation*}
\ker(A_-^*)=\spann\{x^{\omega_+},x^{\omega_-}\}\quad\text{and}\quad\ker(A_+^*)=\spann\{x^{\overline{\omega_+
}},x^{\overline{\omega_-}}\}\:.
\end{equation*}
It can be shown that
\begin{equation} \label{eq:domainsmallg}
\mathcal{D}(A_-^*)=\mathcal{D}(\overline{A_+})\dot{+}\spann\{x^{\overline{\omega_+
+2}},x^{\overline{\omega_-+2}}\}\dot{+}\spann\{x^{\omega_+},x^{\omega_-}\}\;.
\end{equation}
Hence, using Equation \eqref{eq:adjointstar} from Theorem \ref{thm:katocool} and the equality $A_{+,F}=(A_{-,F})^*$, which was shown in Theorem \ref{prop:friedrichsfreude}, we find that the domain of $A_{+,F}$ is given by
\begin{equation*}
\mathcal{D}(A_{+,F})=\mathcal{D}((A_{-,F})^*)=\mathcal{Q}(A_-)\cap\mathcal{D}(A_-^*)=H^1_0(0,1)\cap\mathcal{D}(A_{-}^*)=\mathcal{D}(\overline{A_+})\dot{+}\spann\{x^{\overline{\omega_++2}}-x^{\omega_+},x^{\overline{\omega_-+2}}-x^{\omega_+}\}\:.
\end{equation*}
\end{itemize}
\end{example}
\begin{theorem} \label{prop:festnumrange} Let $A_+=iS+V$ be as in Theorem \ref{prop:friedrichsfreude} and let $A_{+,F}$ be its Friedrichs extension. Then, it is true that $\mathcal{D}(A_{+,F})\subset\mathcal{D}(V_F^{1/2})$ and for any $v\in\mathcal{D}(A_{+,F})$ it holds that
\begin{equation*}
\Real\langle v,A_{+,F}v\rangle=\|V_F^{1/2}v\|^2=\|V_K^{1/2}v\|^2\:.
\end{equation*}
Moreover, this is equivalent to saying that the quadratic form $q$ as defined by
\begin{equation*}
q(v):=\Real\langle v, A_-^*v\rangle-\|V_{K}^{1/2}v\|^2
\end{equation*}
(cf.\ Equation \eqref{eq:simmern}) vanishes identically on $\mathcal{D}(A_{+,F})$, i.e.\
\begin{equation*}
q\upharpoonright_{\mathcal{D}(A_{+,F})}\equiv 0\:.
\end{equation*}
\end{theorem}
\begin{proof} Take any $v\in\mathcal{D}(A_{+,F})$. This means that there exists a sequence $\{v_n\}_{n=1}^\infty\subset\mathcal{D}(A_+)$ such that 
\begin{equation*}
v_n\rightarrow v\qquad\text{and}\qquad\langle v_n,A_+v_n\rangle\rightarrow\langle v,A_{+,F}v\rangle
\end{equation*}
and thus in particular
$$\Real\langle v_n,A_{+,F}v_n\rangle\rightarrow\Real\langle v,A_{+,F}v\rangle\:.$$
We therefore get 
\begin{align*}
\Real\langle v,A_{+,F}v\rangle=\lim_{n\rightarrow\infty}\Real\langle v_n,(iS+V)v_n\rangle=\lim_{n\rightarrow\infty}\langle v_n,Vv_n\rangle\:.
\end{align*}
Now, since $v_n\rightarrow v$ and $\langle v_n,Vv_n\rangle$ converges, this implies that $v\in\mathcal{D}(V_F^{1/2})$ and moreover that 
\begin{equation*}
\|V_F^{1/2}v\|^2=\lim_{n\rightarrow\infty}\langle v_n,Vv_n\rangle=\Real\langle v,A_{+,F}v\rangle\:.
\end{equation*}
Since $V\geq\varepsilon>0$, the equality $$\|V_F^{1/2}v\|^2=\|V_K^{1/2}v\|^2$$ for $v\in\mathcal{D}(V_F^{1/2})$ follows from \eqref{eq:alonsosimon}.
\end{proof}

\section{Proper maximally accretive extensions} \label{sec:proper}
In the following, we are going to develop an analog of the Birman--Kre\u\i n--Vishik theory of selfadjoint extensions, where we want to define a partial order on the real parts of the different extensions of a dual pair of sectorial operators $A_\pm=\pm iS+V$. It turns out that the proper maximally dissipative extensions of $(A_+,A_-)$ can be parametrized by auxiliary operators $D$ that map from a subspace of $\ker A_-^*\cap\mathcal{D}(V_K^{1/2})$ into $\ker A_+^*$. We will denote these extensions by $A_D$. The results in this section can be viewed as a special case of the results in \cite{FNW}, where the more general case of maximally accretive proper extensions of a given dual pair of accretive operators was considered. However, there are two additions we are able to make to the previously obtained results:

Firstly, Theorem \ref{thm:dido} contains a necessary and sufficient condition for $A_D$ to be maximally dissipative even if $\dim(\mathcal{D}(A_D)/\mathcal{D}(A))=\infty$. 

Secondly, the way we present our results in this section will allow us to  present our results in a way that immediately shows how Proposition \ref{prop:kreinistfein} generalizes to the case we are considering in this paper.   
 
In particular, in Sections \ref{sec:form} and \ref{sec:order} we will show that --- provided it is closable --- the closure of the quadratic form $f\mapsto\Real\langle f,A_Df\rangle$ corresponds to a non-negative selfadjoint extension of the real part $V$. This enables us to apply the results of Birman--Kre\u\i n--Vishik in order to define an order between the real parts of the extensions $A_D$. 

In order to present our result in a way similar to Proposition \ref{prop:kreinistfein}, we  introduce the following modified sesquilinear form.
\begin{definition} \label{def:assumption}
Let $(A_+,A_-)$ be a dual pair, where $A_\pm=\pm iS+V$ with $\mathcal{D}(A_\pm)=\mathcal{D}(S)=\mathcal{D}(V)$, where $S$ and $V\geq \varepsilon>0$ are both symmetric and such that $A_\pm$ are sectorial. Then, we define the following non-Hermitian sesquilinear form:
\begin{align*}
[\cdot,\cdot]:\qquad\mathcal{D}(V_K^{1/2})&\times\mathcal{H}\rightarrow\C\\
[f,g]&:=\langle f,g \rangle-2\langle V_K^{1/2}f,V_K^{1/2}A_{+,F}^{-1}g\rangle
\end{align*}
By Theorem \ref{prop:festnumrange}, we have $\mathcal{D}(A_{+,F})\subset\mathcal{D}(V_F^{1/2})\subset\mathcal{D}(V_K^{1/2})$, which means that $[\cdot,\cdot]$ is well-defined. Moreover, for any subset $\mathcal{A}\subset\mathcal{D}(V_K^{1/2})$, let us define its {\bf{orthogonal companion}}  $\mathcal{A}^{[\perp]}$ as
\begin{equation*}
\mathcal{A}^{[\perp]}:=\{g\in\mathcal{H}: [f,g]=0\:\:\text{for all}\:\: f\in\mathcal{A}\}\:.
\end{equation*}
\end{definition}
\begin{remark} \normalfont It is not necessary to compute $V_K^{1/2}$ explicitly in order to determine $[f,g]$, as it is sufficient to know the action of the quadratic form $\psi\mapsto\|V_K^{1/2}\psi\|^2$. The value of $\langle V_K^{1/2}f,V_K^{1/2}A_{+,F}^{-1}g\rangle$ can then be obtained by polarization.
\end{remark}
Let us now use auxiliary operators $D$ from $\ker A_-^*\cap\mathcal{D}(V_K^{1/2})$ to $\ker A_+^*$ and subspaces $\mathfrak{N}$ of $(\ker A_+^*\cap\mathcal{D}(D)^{[\perp]})$ in order to parametrize the proper accretive extensions of $(A_+,A_-)$:
\begin{theorem} Let $(A_+,A_-)$ be a dual pair of sectorial operators of the form as described in Definition \ref{def:assumption}. Then all proper accretive extensions of $(\overline{A_+},A_-)$ can be described by all  pairs of the form $(D,\mathfrak{N})$, where
\begin{itemize}
\item $D$ is an operator from $\ker A_-^*\cap \mathcal{D}(V_K^{1/2})$ to $\ker A_+^*$ that satisfies
\begin{equation} \label{eq:dissipativity}
\Real [\widetilde{k},D\widetilde{k}]\geq\|V_K^{1/2}\widetilde{k}\|^2\quad\text{for all}\:\:\widetilde{k}\in\mathcal{D}(D)
\end{equation}
\item $\mathfrak{N}\subset \left(\ker A_+^*\cap\mathcal{D}(D)^{[\perp]}\right)$ is a linear subspace.
\end{itemize}
The corresponding accretive extensions are described by
\begin{align*}
A_{D,\mathfrak{N}}:\qquad\mathcal{D}(A_{D,\mathfrak{N}})&=\mathcal{D}(\overline{A}_+)\dot{+}\{A_{+,F}^{-1}D\widetilde{k}+\widetilde{k}: \widetilde{k}\in\mathcal{D}(D)\}\dot{+}\{A_{+,F}^{-1}k: k\in\mathfrak{N}\}\\
A_{D,\mathfrak{N}}&=A_-^*\upharpoonright_{\mathcal{D}(A_{D,\mathfrak{N}})}
\end{align*}
Moreover, $A_{D,\mathfrak{N}}$ is maximally accretive if and only if $\mathfrak{N}=\left(\ker A_+^*\cap\mathcal{D}(D)^{[\perp]}\right)$ and $D$ is maximal in the sense that there exists no extension $D\subset D'$ such that $\ker A_+^*\cap\mathcal{D}(D)^{[\perp]}=\ker A_+^*\cap\mathcal{D}(D')^{[\perp]}$ and $D'$ still satisfies \eqref{eq:dissipativity}.
\label{thm:dido}
\end{theorem} 
\begin{proof} Corollary \ref{coro:domain} implies that
\begin{align*}
\mathcal{D}(A_-^*)=\mathcal{D}(\overline{A}_+)\dot{+} A_{+,F}^{-1}\ker (A_+^*)\dot{+}\ker(A_-^*)\:,
\end{align*}
which means that we can choose $\mathcal{D}(A_-^*)//\mathcal{D}(\overline{A}_+)=A_{+,F}^{-1}\ker (A_+^*)\dot{+}\ker(A_-^*)$. Now, all possible subspaces of $A_{+,F}^{-1}\ker (A_+^*)\dot{+}\ker(A_-^*)$ will be of the form $$\mathcal{V}_{D,\mathfrak{N}}:=\{A_{+,F}^{-1}D\widetilde{k}+\widetilde{k}: \widetilde{k}\in\mathcal{D}(D)\}\dot{+}\{A_{+,F}^{-1}k: k\in\mathfrak{N}\}\:,$$
where $D$ is a linear map from $\mathcal{D}(D)\subset\ker(A_-^*)$ to $\ker (A_+^*)$ and $\mathfrak{N}\subset\ker (A_+^*)$. We therefore use the pairs $(D,\mathfrak{N})$ to parametrize all proper extensions of the dual pair $(A_+,A_-)$ via $A_{+,\mathcal{V}_{D ,\mathfrak{N}}}=:A_{D,\mathfrak{N}}$. Thus, by Proposition \ref{prop:accretivetelemann}, $A_{D,\mathfrak{N}}$ is accretive if and only if \\(i): $$\mathcal{V}_{D,\mathfrak{N}}\subset\mathcal{D}(V_K^{1/2})$$ and\\ (ii):
\begin{align} \label{eq:susan}
q((A_{+,F}^{-1}D\widetilde{k}+\widetilde{k})+A_{+,F}^{-1}k)&:=\notag\\\Real\langle (A_{+,F}^{-1}D\widetilde{k}&+\widetilde{k})+A_{+,F}^{-1}k,A_-^*[(A_{+,F}^{-1}D\widetilde{k}+\widetilde{k})+A_{+,F}^{-1}k]\rangle\notag\\&\qquad\qquad\qquad-\|V_K^{1/2}(A_{+,F}^{-1}D\widetilde{k}+\widetilde{k}+A_{+,F}^{-1}k)\|^2\geq 0\:,
\end{align}
for all $\widetilde{k}\in\mathcal{D}(D)$ and $k\in\mathfrak{N}$. By Theorem \ref{prop:festnumrange} we have that $\mathcal{D}(A_{+,F})\subset\mathcal{D}(V_F^{1/2})\subset\mathcal{D}(V_K^{1/2})$, which means that Condition (i) is satisfied if and only if $\mathcal{D}(D)\subset\mathcal{D}(V_K^{1/2})$. Now, let us rewrite \eqref{eq:susan}:
\begin{align*}
&q((A_{+,F}^{-1}D\widetilde{k}+\widetilde{k})+A_{+,F}^{-1}k)\\
&=\Real\langle (A_{+,F}^{-1}D\widetilde{k}+\widetilde{k})+A_{+,F}^{-1}k,A_-^*[(A_{+,F}^{-1}D\widetilde{k}+\widetilde{k})+A_{+,F}^{-1}k]\rangle-\|V_K^{1/2}(A_{+,F}^{-1}D\widetilde{k}+\widetilde{k}+A_{+,F}^{-1}k)\|^2\\
&=\Real\langle A_{+,F}^{-1}D\widetilde{k}+A_{+,F}^{-1}k,A_{+,F}[A_{+,F}^{-1}D\widetilde{k}+A_{+,F}^{-1}k]\rangle+\Real\langle \widetilde{k},D\widetilde{k}+k\rangle\\&-\|V_K^{1/2}(A_{+,F}^{-1}D\widetilde{k}+A_{+,F}^{-1}k)\|^2-\|V_K^{1/2}\widetilde{k}\|^2-2\Real\langle V_K^{1/2}\widetilde{k},V_K^{1/2}(A_{+,F}^{-1}D\widetilde{k}+A_{+,F}^{-1}k)\rangle\\
&=\Real\left(\langle \widetilde{k},D\widetilde{k}\rangle-2\langle V_K^{1/2}\widetilde{k},V_K^{1/2}A_{+,F}^{-1}D\widetilde{k}\rangle\right)\\&+\Real\left(\langle\widetilde{k},k\rangle-2\langle V_K^{1/2}\widetilde{k},V_K^{1/2}A_{+,F}^{-1}k\rangle\right)-\|V_K^{1/2}\widetilde{k}\|^2\\
&=\Real[\widetilde{k},D\widetilde{k}]+\Real[\widetilde{k},k]-\|V_K^{1/2}\widetilde{k}\|^2\geq 0\:,
\end{align*}
where we have used that by Theorem \ref{prop:festnumrange},
\begin{equation}\label{eq:agar}
\Real\langle A_{+,F}^{-1}D\widetilde{k}+A_{+,F}^{-1}k,A_{+,F}[A_{+,F}^{-1}D\widetilde{k}+A_{+,F}^{-1}k]\rangle=\|V_K^{1/2}(A_{+,F}^{-1}D\widetilde{k}+A_{+,F}^{-1}k)\|^2\:.
\end{equation}
Now, assume 
\begin{equation} \label{eq:condition}
\mathfrak{N}\subset (\ker (A_+^*)\cap\mathcal{D}(D)^{[\perp]})\quad \text{and} \quad \Real[\widetilde{k},D\widetilde{k}]-\|V_K^{1/2}\widetilde{k}\|^2\geq 0 \quad\text{for all}\quad \widetilde{k}\in\mathcal{D}(D)\:.
\end{equation}
 Hence, we get that
\begin{equation*}
q((A_{+,F}^{-1}D\widetilde{k}+\widetilde{k})+A_{+,F}^{-1}k)=\Real[\widetilde{k},D\widetilde{k}]+\Real[\widetilde{k},k]-\|V_K^{1/2}\widetilde{k}\|^2=\Real[\widetilde{k},D\widetilde{k}]-\|V_K^{1/2}\widetilde{k}\|^2\geq 0\:,
\end{equation*}
for all $\widetilde{k}\in\mathcal{D}(D)$ and all $k\in\mathfrak{N}$. This means that Condition \eqref{eq:condition} being satisfied is sufficient for $A_{D,\mathfrak{N}}$ to be accretive. Let us now show that it is also necessary. Thus, assume that Condition \eqref{eq:condition} is not satisfied. If there exists a $\widetilde{k}\in\mathcal{D}(D)$ such that $\Real[\widetilde{k},D\widetilde{k}]-\|V_K^{1/2}\widetilde{k}\|^2<0$, this means that \eqref{eq:susan} cannot be satisfied in this case as we can choose $k=0$. Moreover, if there exists a $k\in\mathfrak{N}$ and a $\widetilde{k}\in\mathcal{D}(D)$ such that $[\widetilde{k},k]\neq 0$, this means that we can replace $k\mapsto \lambda k$, where $\lambda\in\C$ is suitably chosen such that
$$ q((A_{+,F}^{-1}D\widetilde{k}+\widetilde{k})+A_{+,F}^{-1}\lambda k)=\Real[\widetilde{k},D\widetilde{k}]+\Real[\widetilde{k},\lambda k]-\|V_K^{1/2}\widetilde{k}\|^2< 0\:,$$
which means that $A_{D,\mathfrak{N}}$ cannot be accretive in this case either.\\
Let us now prove that $A_{D,\mathfrak{N}}$ is maximally accretive if and only if $D$ is maximal in the sense as stated in the theorem and $\mathfrak{N}=\ker (A_+^*)\cap\mathcal{D}(D)^{[\perp]}$. Clearly, if there exists a $D\subset D'$ such that $\Real [\widetilde{k},D'\widetilde{k}]\geq \|V_K^{1/2}\widetilde{k}\|^2$ for all $\widetilde{k}\in\mathcal{D}(D')$ and $(\ker (A_+^*)\cap\mathcal{D}(D)^{[\perp]})=(\ker (A_+^*)\cap\mathcal{D}(D')^{[\perp]})$ or a $\mathfrak{N}\subset\mathfrak{N}'\subset (\ker (A_+^*)\cap\mathcal{D}(D)^{[\perp]})$, we get that $A_{D',\mathfrak{N}'}$ is an accretive extension of $A_{D,\mathfrak{N}}$. 

For the other direction, let us assume that $A_{D,\mathfrak{N}}$ is not maximally accretive. It is clear that the operator $A_{D, \widehat{\mathfrak{N}}}$, where $\widehat{\mathfrak{N}}=\ker A^*\cap\mathcal{D}(D)^{[\perp]}$, is an accretive extension of $A_{D,\mathfrak{N}}$ and from now on, we will therefore only consider this case. By Proposition \ref{thm:friendly}, we know that there exists a proper maximally accretive extension $\widehat{A}$ of the dual pair $(A_{D,\widehat{\mathfrak{N}}},A_-)$ and thus a subspace $\widehat{\mathcal{V}}\subset A_{+,F}^{-1}\ker (A_+^*)\dot{+}\ker (A_-^*)$ such that $\widehat{A}=A_{+,\widehat{\mathcal{V}}}$. Moreover, by what we have shown above, there exists an operator $D'$ that satisfies \eqref{eq:dissipativity} and a subspace $\mathfrak{N}'\subset(\ker (A_+^*)\cap\mathcal{D}(D')^{[\perp]})$ such that $A_{D,\widehat{\mathfrak{N}}}\subset\widehat{A}=A_{D',\mathfrak{N}'}$. Note that $A_{D,\widehat{\mathfrak{N}}}\subset A_{D',\mathfrak{N}'}$ is equivalent to $\mathcal{V}_{D,\widehat{\mathfrak{N}}}\subset\mathcal{V}_{D',\mathfrak{N}'}$, which means that
\begin{equation} \label{eq:codomains}
\{A_{+,F}^{-1}D\widetilde{k}+\widetilde{k}:\widetilde{k}\in\mathcal{D}(D)\}\dot{+}\{A_{+,F}^{-1}k:k\in\widehat{\mathfrak{N}}\}\subset\{A_{+,F}^{-1}D'\widetilde{k}'+\widetilde{k}':\widetilde{k}'\in\mathcal{D}(D')\}\dot{+}\{A_{+,F}^{-1}k':k'\in{\mathfrak{N}'}\}\:.
\end{equation}
 Let us now show that this implies that $\widehat{\mathfrak{N}}\subset\mathfrak{N}'$ as well as $\mathcal{D}(D)\subset\mathcal{D}(D')$. To begin with, assume that $\widehat{\mathfrak{N}}\not\subset\mathfrak{N}'$, i.e.\ that there exists a $k\in\widehat{\mathfrak{N}}$ such that $k\notin\mathfrak{N}'$. By \eqref{eq:codomains}, this means that there exists a $0\neq\widetilde{k}'\in\mathcal{D}(D')$ and a $k'\in\mathfrak{N}'$ such that 
 $$ A_{+,F}^{-1}k=A_{+,F}^{-1}D'\widetilde{k}'+\widetilde{k}'+A_{+,F}^{-1}k'\:,$$
which is equivalent to
 $$ A_{+,F}^{-1}(k-D'\widetilde{k}'-k')=\widetilde{k}'\:.$$
 However, note that the left hand side of this equation is an element of $A_{+,F}^{-1}\ker(A_+^*)$, while the right hand side is an element of $\ker(A_-^*)$. Since $A_{+,F}^{-1}\ker(A_+^*)\cap \ker(A_-^*)=\{0\}$, this implies that $\widetilde{k}'=0$ --- a contradiction. Hence $\widehat{\mathfrak{N}}\subset\mathfrak{N}'$. It can be argued analogously that $\mathcal{D}(D)\subset\mathcal{D}(D')$. Assume that this is not the case, i.e.\ that there exists a $\widetilde{k}\in\mathcal{D}(D)$ such that $\widetilde{k}\notin\mathcal{D}(D')$. By \eqref{eq:codomains}, it is true that there exists a $0\neq k'\in\mathfrak{N}'$ and a $\widetilde{k}'\in\mathcal{D}(D')$ such that 
 $$ A_{+,F}^{-1}D\widetilde{k}+\widetilde{k}=A_{+,F}^{-1}D'\widetilde{k}'+\widetilde{k}'+A_{+,F}^{-1}k'\:,$$
 being equivalent to
$$
A_{+,F}^{-1}(D\widetilde{k}-D'\widetilde{k}'-k')=\widetilde{k}'-\widetilde{k}\:, 
 $$
 from which again by a reasoning similar to above, we infer that $\widetilde{k}'-\widetilde{k}=0$. This however implies that $A_{+,F}^{-1}k'=0$ which implies that $k'=0$ since $A_{+,F}^{-1}$ is injective. But $k'=0$ is a contradiction from which we conclude that $\mathcal{D}(D)\subset\mathcal{D}(D')$.
 
We thus have shown that $\mathcal{D}(D)\subset \mathcal{D}(D')$ and $(\ker (A_+^*)\cap\mathcal{D}(D)^{[\perp]})=\widehat{\mathfrak{N}}\subset\mathfrak{N}'\subset (\ker (A_+^*)\cap\mathcal{D}(D')^{[\perp]})$. Now, $\mathcal{D}(D)\subset\mathcal{D}(D')$ implies that $\mathcal{D}(D)^{[\perp]}\supset\mathcal{D}(D')^{[\perp]}$, from which it immediately follows that $(\ker (A_+^*)\cap\mathcal{D}(D)^{[\perp]})\supset(\ker (A_+^*)\cap\mathcal{D}(D')^{[\perp]})$. This implies that $(\ker (A_+^*)\cap\mathcal{D}(D)^{[\perp]})=(\ker (A_+^*)\cap\mathcal{D}(D')^{[\perp]})$, which shows that $D$ was not maximal in the sense as stated in the theorem.
\end{proof}
Let us make a few remarks at this point:
\begin{remark} \normalfont \label{rem:haydn} Note that the correspondence between the pairs $(D,\mathfrak{N})$ and the proper maximally accretive extensions $A_{D,\mathfrak{N}}$ of $(A_+,A_-)$ is not one-to-one. This follows from the fact that for any $\widetilde{k}\in\mathcal{D}(D)$ and any $k\in(\ker (A_+^*)\cap\mathcal{D}(D)^{[\perp]})$, we can write
\begin{equation*}
\spann\{A_{+,F}^{-1}D\widetilde{k}+\widetilde{k}\}\dot{+}\spann\{A_{+,F}^{-1}k\}=\spann\{A_{+,F}^{-1}(D\widetilde{k}+ k)+\widetilde{k}\}\dot{+}\spann\{A_{+,F}^{-1}k\}\:,
\end{equation*}
i.e. if we have an auxiliary operator $D'$ with $\mathcal{D}(D')=\mathcal{D}(D)$ and $D'\widetilde{k} - D\widetilde{k}\in(\ker (A_+^*)\cap\mathcal{D}(D)^{[\perp]})$ for all $\widetilde{k}\in(\ker (A_+^*)\cap\mathcal{D}(D)^{[\perp]})$ we would get that $A_{D',\mathfrak{N}}=A_{D,\mathfrak{N}}$.
However, we could for example restrict our considerations to auxiliary operators that satisfy $D\widetilde{k}\perp k$ for all $\widetilde{k}\in\mathcal{D}(D)$ and all $k\in(\ker (A_+^*)\cap\mathcal{D}(D)^{[\perp]})$.
With this additional requirement, the correspondence between $(D,\mathfrak{N})$ and proper accretive extensions $A_{D,\mathfrak{N}}$ of $(A_+,A_-)$ becomes one-to-one.
\end{remark}
\begin{remark} \normalfont For the case that $\mathcal{D}(D)$ is finite-dimensional the maximality condition on $D$ is automatically satisfied. In this case, $A_{D,\mathfrak{N}}$ is therefore maximally accretive if and only if $\mathfrak{N}=(\ker (A_+^*)\cap\mathcal{D}(D)^{[\perp]})$.
\end{remark}
\begin{remark} \normalfont Observe that this theorem reduces to the result of Proposition \ref{prop:kreinistfein} in the case of (maximally) accretive extensions of a dual pair of strictly positive symmetric operators $A_\pm=V$, since 
\begin{equation*}
\Real[\widetilde{k},D\widetilde{k}]=\Real \langle\widetilde{k},D\widetilde{k}\rangle-2\Real\langle \underbrace{V_K^{1/2}\widetilde{k}}_{=0},V_K^{1/2}V_F^{-1}\widetilde{k}\rangle=\Real \langle\widetilde{k},D\widetilde{k}\rangle\:,
\end{equation*}
where $V_K^{1/2}\widetilde{k}=0$ follows from the fact that $\widetilde{k}\in \ker(A_\pm^*)=\ker(V^*)=\ker(V_K)=\ker(V_K^{1/2})$.
This is of course the same as requiring that $D$ be an accretive operator from $\mathcal{D}(D)\subset\ker(V^*)$ into $\ker(V^*)$. The condition that $V_D$ is maximally accretive if and only if $D$ is a maximally accretive operator in $\overline{\mathcal{D}(D)}$ follows also from this theorem since $(\ker V^*\cap\mathcal{D}(D)^\perp)=(\ker V^*\cap\mathcal{D}(D')^\perp)$ is equivalent to $\overline{\mathcal{D}(D)}=\overline{\mathcal{D}(D')}$.
\end{remark}
\begin{convention} If $A_{D,\mathfrak{N}}$ is a maximally accretive extension of the dual pair $(A_+,A_-)$ as defined in Theorem \ref{thm:dido}, we know that $\mathfrak{N}$ is determined by the choice of $D$: $\mathfrak{N}=\ker (A_+^*)\cap\mathcal{D}(D)^{[\perp]}$. Thus, if $A_{D,\mathfrak{N}}$ is maximally accretive, let us just write $A_D$ instead of $A_{D,\mathfrak{N}}$.
\end{convention}
\begin{example} \normalfont \label{ex:mabaker} Let $\gamma>0$ and $\mathcal{H}=L^2(0,1)$ and consider the dual pair $A_\pm=\pm i\frac{\gamma}{x^2}-\frac{\text{d}^2}{\text{d}x^2}$ with domain $\mathcal{C}_c^\infty(0,1)$. Then the real part of $A_\pm$ is given by $V=-\frac{\text{d}^2}{\text{d}x^2}$ with domain $\mathcal{C}_c^\infty(0,1)$. It can be shown that the domain of $V_K^{1/2}$ is given by $H^1(0,1)$ and since $\ker V^*=\spann\{1,x\}$, we get for any $f\in H^1(0,1)$:
\begin{equation*}
\|V_K^{1/2}f\|^2=\|V_F^{1/2}(f(x)-(1-x)f(0)-xf(1))\|^2=\|f'\|^2-|f(1)-f(0)|^2\:.
\end{equation*}

Define the numbers $\omega_\pm:=\frac{1\pm\sqrt{1+4i\gamma}}{2}$. We now have to distinguish between the two cases $\gamma\geq\sqrt{3}$ and $0<\gamma<\sqrt{3}$ as the dimension of $\ker(A_\pm^*)$ is different in either case:
\begin{itemize}
\item For $\gamma \geq\sqrt{3}$: $\ker A_-^*=\spann\{x^{\omega_+}\}$  and  $\ker A_+^*=\spann\{x^{\overline{\omega_+}}\}$. Moreover, $x^{\omega_+}\in H^1(0,1)=\mathcal{D}(V_K^{1/2})$ from which we get that $(\ker A_-^*\cap\mathcal{D}(V_K^{1/2}))=\spann\{x^{\omega_+}\}$.
There are two possibilities for the choice of $\mathcal{D}(D)$. We may either choose $\mathcal{D}(D)=\{0\}$ which parametrizes the Friedrichs extension $A_{+,F}$ of $A_+$. The second possibility is the choice $\mathcal{D}(D)=\spann\{x^{\omega_+}\}$. Observe that any map from $\spann\{x^{\omega_+}\}$ into $\spann\{x^{\overline{\omega_+}}\}$ has to be of the form
\begin{equation} \label{eq:eriugena}
Dx^{\omega_+}=dx^{\overline{\omega_+}}\:,
\end{equation}
where $d\in\C$.
A calculation shows that
\begin{align}[x^{\omega_+},Dx^{\omega_+}]&=d[x^{\omega_+},x^{\overline{\omega_+}}]\notag\\&=d\left(\frac{1}{2\overline{\omega_+}+1}-\frac{2}{i\gamma-(\overline{\omega_+}+2)(\overline{\omega_+}+1)}\left[\frac{\overline{\omega_+}(\overline{\omega_+}+2)}{2\overline{\omega_+}+1}-\frac{|{\omega_+}|^2}{\omega_++\overline{\omega_+}-1}\right]\right)=d\cdot\sigma(\omega_+)\:, \label{eq:scotus}
\end{align}
where we have defined 

\begin{equation} \sigma(\omega_+):=\frac{1}{2\overline{\omega_+}+1}-\frac{2}{i\gamma-(\overline{\omega_+}+2)(\overline{\omega_+}+1)}\left[\frac{\overline{\omega_+}(\overline{\omega_+}+2)}{2\overline{\omega_+}+1}-\frac{|{\omega_+}|^2}{\omega_++\overline{\omega_+}-1}\right]\:.
\label{eq:thomismus}
\end{equation}
From another calculation, we get
\begin{equation} \label{eq:bonaventura}
\|V_K^{1/2}x^{\omega_+}\|^2=\|(x^{\omega_+})'\|^2-1=\frac{|\omega_+-1|^2}{\omega_++\overline{\omega_+}-1}=:\tau(\omega_+)\:.
\end{equation}
Thus, all proper maximally accretive extensions of $(A_+,A_-)$ that are different from $A_{+,F}$ are given by
\begin{align*}
A_D:\qquad\mathcal{D}(A_D)&=\mathcal{D}(\overline{A_+})\dot{+}\spann\{A_{+,F}^{-1}Dx^{\omega_+}+x^{\omega_+}\}=\mathcal{D}(\overline{A_+})\dot{+}\spann\{A_{+,F}^{-1}dx^{\overline{\omega_+}}+x^{\omega_+}\}\\
A_D&=A_-^*\upharpoonright_{\mathcal{D}(A_D)}\:,
\end{align*}
where $d\in\C$ has to lie in the half-plane of the complex plane given by
\begin{align} \label{eq:bkvcondition}
\Real [x^{\omega_+},dx^{\overline{\omega_+}}]&\geq\|V_K^{1/2}x^{\omega_+}\|^2\notag\\
\Leftrightarrow\qquad\Real(d\sigma(\omega_+))&\geq \tau(\omega_+)\:.
\end{align}

\item For $0<\gamma<\sqrt{3}$: $\ker (A_-^*)=\spann\{x^{\omega_+},x^{\omega_-}\}$  and  $\ker(A_+^*)=\spann\{x^{\overline{\omega_+}},x^{\overline{\omega_-}}\}$. However, we have $x^{\omega_-}\notin H^1(0,1)=\mathcal{D}(V_K^{1/2})$, since $\Real (\omega_-)<1/2$. This means that $(\ker A_-^*\cap\mathcal{D}(V_K^{1/2}))=\spann\{x^{\omega_+}\}$. Hence, the only two choices for $\mathcal{D}(D)$ are again either $\mathcal{D}(D)=\{0\}$ which corresponds to the Friedrichs extension $A_{+,F}$ of $A_+$ or $\mathcal{D}(D)=\spann\{x^{\omega_+}\}$. As we have already described $A_{+,F}$ in \eqref{eq:domainsmallg}, let us now focus on the case $\mathcal{D}(D)=\spann\{x^{\omega_+}\}$. 

To this end, let us determine $(\mathcal{D}(D)^{[\perp]}\cap\ker (A_+^*))=(\spann\{x^{\omega_+}\}^{[\perp]}\cap\spann\{x^{\overline{\omega_+}},x^{\overline{\omega_-}}\})$, which means that we have to find the solution space of
$$ \langle x^{\omega_+},\lambda x^{\overline{\omega_+}}+\mu x^{\overline{\omega_-}}\rangle-2\langle V_K^{1/2}x^{\omega_+},V_K^{1/2}A_{+,F}^{-1}(\lambda x^{\overline{\omega_+}}+\mu x^{\overline{\omega_-}})\rangle=0\:,$$
which is given by $\spann\{A_{+,F}\chi\}$, where
 $$\chi(x):=(2+\overline{\omega_-}-\omega_+)(x^{\overline{\omega_+}+2}-x^{\omega_+})+(\omega_+-\overline{\omega_+}-2)(x^{\overline{\omega_-}+2}-x^{\omega_+})$$
and thus 
$$ \left(A_{+,F}\chi\right)(x)=\frac{2+\overline{\omega_-}-\omega_+}{ i\gamma-(\overline{\omega_+}+2)(\overline{\omega_+}+1)}x^{\overline{\omega_+}}+\frac{\omega_+-\overline{\omega_+}-2}{i\gamma-(\overline{\omega_-}+2)(\overline{\omega_-}+1)}x^{\overline{\omega_-}}\:.$$
Next, let us argue that it is also sufficient to only consider maps $D$ of the form \eqref{eq:eriugena}. This follows from what has been said in Remark \ref{rem:haydn}. To see this, assume that the map $D$ is of the form 
$$Dx^{\omega_+}=d_+x^{\overline{\omega_+}}+d_-x^{\overline{\omega_-}}$$
for some numbers $d_+,d_-\in\C$.
Then, since $(A_{+,F}^{-1}x^{\overline{\omega_\pm}})\propto (x^{\overline{\omega_\pm}+2}-x^{\omega_+})$ we can find numbers $\lambda, \mu\in\C$ such that $$(A_{+,F}^{-1}Dx^{\omega_+}+\lambda\chi(x))=\mu(x^{\overline{\omega_+}+2}-x^{\omega_+})\propto A_{+,F}^{-1}x^{\overline{\omega_+}}\:,$$
which means that there exists another number $\nu\in\C$ such that
\begin{equation} \label{eq:bonteededieu}
(A_{+,F}^{-1}Dx^{\omega_+}+\lambda\chi(x))=A_{+,F}^{-1}(Dx^{\omega_+}+\lambda A_{+,F}\chi(x))=\nu A_{+,F}^{-1}x^{\overline{\omega_+}}\:. 
\end{equation}
 Thus, the operator $D'$ given by
$$ D'x^{\omega_+}= Dx^{\omega_+}+\lambda A_{+,F}\chi(x)$$
maps $\spann\{x^{\omega_+}\}$ into $\spann\{x^{\overline{\omega_+}}\}$, which follows from \eqref{eq:bonteededieu} and the fact that $A_{+,F}^{-1}$ is injective. Moreover, since $A_{+,F}\chi\in(\ker A_+^*\cap\mathcal{D}(D)^{[\perp]})$, we have that $A_D=A_{D'}$. Hence, for any $0<\gamma<\sqrt{3}$, it also suffices to consider only auxiliary operators of the form \eqref{eq:eriugena}. Moreover, since the calculations in \eqref{eq:scotus} and \eqref{eq:thomismus} do not change for $0<\gamma<\sqrt{3}$, we get the following description of the proper accretive extensions of $(A_+,A_-)$ that are different to $A_{+,F}$:
\begin{align*}
A_D:\qquad\mathcal{D}(A_D)&=\mathcal{D}(\overline{A_+})\dot{+}\spann\{A_{+,F}^{-1}Dx^{\omega_+}+x^{\omega_+}\}\dot{+}\spann\{\chi\}\\&=\mathcal{D}(\overline{A_+})\dot{+}\spann\{A_{+,F}^{-1}dx^{\overline{\omega_+}}+x^{\omega_+}\}\dot{+}\spann\{\chi\}\\
A_D&=A_-^*\upharpoonright_{\mathcal{D}(A_D)}\:,
\end{align*}
where $d\in\C$ has to again satisfy Condition \eqref{eq:bkvcondition}.
\end{itemize}

\end{example}
\section{Closability of the quadratic form induced by the real part}
\label{sec:form} 
Next, we want to investigate what can be said about the quadratic form associated to the real part of a maximally accretive extension of a dual pair $(A_+,A_-)$ satisfying the assumptions of Theorem \ref{thm:dido}. We start by defining this quadratic form.
\begin{definition} Let $(A_+,A_-)$ be a dual pair satisfying the assumptions of Theorem \ref{thm:dido}. For any proper maximally accretive extension $A_D$ of $(A_+,A_-)$ let us define the associated non-negative quadratic form $\mathfrak{re}_{D,0}$:
\begin{align*}
\mathfrak{re}_{D,0}:\qquad\mathcal{D}(\mathfrak{re}_{D,0})&=\mathcal{D}(A_D)\\
\mathfrak{re}_{D,0}(\psi)&=\Real\langle \psi,A_D\psi\rangle\:.
\end{align*}
Moreover, if $\mathfrak{re}_{D,0}$ is closable let us denote its closure by
$\mathfrak{re}_D$. Recall that $\mathfrak{re}_D$ is given by:
\begin{align*}
&\mathfrak{re}_D:\\&\mathcal{D}(\mathfrak{re}_D)=\{f\in\mathcal{H}: \exists \{f_n\}_{n=1}^\infty\subset\mathcal{D}(\mathfrak{re}_{D,0})\:\text{s.t.}\: \|f_n-f\|\overset{n\rightarrow\infty}{\longrightarrow}0\:\text{and}\:\|f_n-f_m\|_{\mathfrak{re}_D}\overset{n,m\rightarrow\infty}{\longrightarrow}0\}\\
&\mathfrak{re}_{D}(f):=\lim_{n\rightarrow\infty}\mathfrak{re}_{D,0}(f_n)\:,
\end{align*}
where $\|\cdot\|_{\mathfrak{re}_D}$ denotes the norm induced by $\mathfrak{re}_{D,0}$: $$\|f\|_{\mathfrak{re}_D}^2:=\|f\|^2+\mathfrak{re}_{D,0}(f)=\|f\|^2+\Real\langle f,A_Df\rangle\quad\text{for all}\:\: f\in\mathcal{D}(A_D)\:.$$
 Moreover, let us denote the non-negative selfadjoint operator associated to $\mathfrak{re}_D$ by $V_D$.
\end{definition}
By \cite[Thm.\ VI, 1.27]{Kato} each non-negative selfadjoint operator $\widehat{V}$ induces a closable quadratic form.  However, it is not always the case that the form $\mathfrak{re}_{D,0}$ is closable. Let us now give a necessary and sufficient condition for $\mathfrak{re}_{D,0}$ to be closable. For technical reasons that will become evident during the proof, we have to restrict ourselves to the case $\dim(D)<\infty$.
\begin{theorem} \label{thm:closable}
Let $A_D$ be defined as in Theorem \ref{thm:dido} and assume that $V\geq\varepsilon>0$ as well as $\dim\mathcal{D}(D)<\infty$. Then, $\mathfrak{re}_{D,0}$ is closable if and only if
\begin{equation} \label{eq:zeitgenug}
q(A_{+,F}^{-1}D\widetilde{k}+\widetilde{k})=\Real [\widetilde{k},D\widetilde{k}]-\|V_K^{1/2}\widetilde{k}\|^2=0\quad\text{for all}\quad \widetilde{k}\in\mathcal{D}(D)\cap\mathcal{D}(V_F^{1/2})\:.
\end{equation}
\end{theorem}
\begin{proof} Firstly, let us show that \eqref{eq:zeitgenug} is necessary for $\mathfrak{re}_{D,0}$ to be closable. Thus, assume that there exists a $\widetilde{k}\in\mathcal{D}(D)\cap\mathcal{D}(V_F^{1/2})$ such that $q(A_{+,F}^{-1}D\widetilde{k}+\widetilde{k})\neq 0$. Since by Theorem \ref{prop:festnumrange}, we have that $A_{+,F}^{-1}D\widetilde{k}+\widetilde{k}\in\mathcal{D}(V_F^{1/2})$, there exists a sequence $\{f_n\}_{n=1}^\infty\subset\mathcal{D}(V)$ that is Cauchy with respect to $\|\cdot\|^2+\langle\cdot,V\cdot\rangle$ such that
\begin{equation*}
\|f_n+A^{-1}_{+,F}D\widetilde{k}+\widetilde{k}\|^2+\|V_F^{1/2}(f_n+A^{-1}_{+,F}D\widetilde{k}+\widetilde{k})\|^2\overset{n\rightarrow\infty}{\longrightarrow} 0\:.
\end{equation*}
This means in particular that the sequence $g_n:=f_n+A_{+,F}^{-1}D\widetilde{k}+\widetilde{k}$ converges to $0$:
\begin{equation*}
\lim_{n\rightarrow\infty}\|g_n\|=0\:.
\end{equation*}
Also, since $\{f_n\}_{n=1}^\infty\subset\mathcal{D}(V)\subset\mathcal{D}(V_F^{1/2})$, we can show that $\{g_n\}_{n=1}^\infty$ is Cauchy with respect to $\|\cdot\|_{\mathfrak{re}_D}$:
\begin{equation*}
\|g_n-g_m\|_{\mathfrak{re}_D}^2=\|f_n-f_m\|_{\mathfrak{re}_D}^2=\|f_n-f_m\|^2+\|V_F^{1/2} (f_n-f_m)\|^2\overset{n,m\rightarrow\infty}{\longrightarrow} 0\:.
\end{equation*}
However, by Lemma \ref{lemma:cernohorsky} we have that
\begin{align*}
\Real\langle g_n,A_Dg_n\rangle&=\Real\langle f_n+A_{+,F}^{-1}D\widetilde{k}+\widetilde{k},A_D(f_n+A_{+,F}^{-1}D\widetilde{k}+\widetilde{k})\rangle\\&=\|V_K^{1/2}(f_n+A_{+,F}^{-1}D\widetilde{k}+\widetilde{k})\|^2+q(A_{+,F}^{-1}D\widetilde{k}+\widetilde{k})\:.
\end{align*} 
Moreover, since $\|V_K^{1/2} (f_n+A_{+,F}^{-1}D\widetilde{k}+\widetilde{k})\|=\|V_F^{1/2} (f_n+A_{+,F}^{-1}D\widetilde{k}+\widetilde{k})\|$, we get
\begin{align*}
\|g_n\|^2_{\mathfrak{re}_D}=\|f_n+A_{+,F}^{-1}D\widetilde{k}+\widetilde{k}\|^2&+\|V_F^{1/2} (f_n+A_{+,F}^{-1}D\widetilde{k}+\widetilde{k})\|^2\\&+q(A_{+,F}^{-1}D\widetilde{k}+\widetilde{k})\overset{n\rightarrow\infty}{\longrightarrow} q(A_{+,F}^{-1}D\widetilde{k}+\widetilde{k}) \neq 0\:,
\end{align*}
which shows that $\mathfrak{re}_{D,0}$ is not closable.\\
Now, let us show that \eqref{eq:zeitgenug} being satisfied implies that $\mathfrak{re}_{D,0}$ is closable. To this end, let us firstly show that $A_D$ being accretive and \eqref{eq:zeitgenug} imply that
\begin{equation} \label{eq:schee}
q(A_{+,F}^{-1}D(\widetilde{k}_1+\widetilde{k}_2)+ \widetilde{k}_1+\widetilde{k}_2)=q(A_{+,F}^{-1}D\widetilde{k}_1+\widetilde{k}_1)
\end{equation}
for any $\widetilde{k}_1\in\mathcal{D}(D)$ and any $\widetilde{k}_2\in\mathcal{D}(D)\cap\mathcal{D}(V_F^{1/2})$. Since $q(A_{+,F}^{-1}D\widetilde{k}_2+\widetilde{k}_2)=0$ by assumption, for any $\lambda\in\C$, we get
\begin{equation} \label{eq:salzburg}
q(A_{+,F}^{-1}D\widetilde{k}_1+\widetilde{k}_1+\lambda(A_{+,F}^{-1}D\widetilde{k}_2+\widetilde{k}_2))=q(A_{+,F}^{-1}D\widetilde{k}_1+\widetilde{k}_1)+2\Real[\lambda q(A_{+,F}^{-1}D\widetilde{k}_1+\widetilde{k}_1,A_{+,F}^{-1}D\widetilde{k}_2+\widetilde{k}_2)]\:,
\end{equation}
where $q(\cdot,\cdot)$ denotes the sesquilinear form associated to $q$. This implies that $q(A_{+,F}^{-1}D\widetilde{k}_1+\widetilde{k}_1,A_{+,F}^{-1}D\widetilde{k}_2+\widetilde{k}_2)=0$, since otherwise, we could choose a suitable $\lambda\in\C$ such that the right hand side of \eqref{eq:salzburg} is negative. This, however, would contradict the accretivity of $A_D$, from which we have $q(A_{+,F}^{-1}D\widetilde{k}_1+\widetilde{k}_1+\lambda(A_{+,F}^{-1}D\widetilde{k}_2+\widetilde{k}_2))\geq 0$.
Next, let us define the operator $\mathcal{P}$ to be the unbounded projection onto $\ker V^*$ along $\mathcal{D}(V_F^{1/2})$ according to the decomposition $\mathcal{D}(V_K^{1/2})=\mathcal{D}(V_F^{1/2})\dot{+}\ker V^*$:
\begin{align} \label{eq:audietis}
\mathcal{P}:\qquad\mathcal{D}(\mathcal{P})=\mathcal{D}(V_K^{1/2})&=\mathcal{D}(V_F^{1/2})\dot{+}\ker V^*\notag\\
\mathcal{P}(v_F+v^*)&=v^*\:,
\end{align}
where $v_F\in\mathcal{D}(V_F^{1/2})$ and $v^*\in\ker V^*$. Moreover, let us define $\mathcal{D}_2:=\mathcal{D}(D)\cap\mathcal{D}(V_F^{1/2})$ and decompose 
\begin{equation} \label{eq:decomposition}
\mathcal{D}(D)=\mathcal{D}_2 \dot{+}\mathcal{D}(D)//\mathcal{D}_2\:.
\end{equation}
Now, let $\{f_n\}_{n=1}^\infty \subset\mathcal{D}(A_D)$ be a sequence that converges to $0$ and that is Cauchy with respect to $\|\cdot\|_{\mathfrak{re}_D}$. In general form, it can be written as
\begin{equation*}
f_n:=f_{0,n}+A_{+,F}^{-1}k_n+A_{+,F}^{-1}D\widetilde{k}_n^{(1)}+\widetilde{k}_n^{(1)}+A_{+,F}^{-1}D\widetilde{k}_n^{(2)}+\widetilde{k}_n^{(2)}\:,
\end{equation*}
where $\{f_{0,n}\}_{n=1}^\infty\subset\mathcal{D}(\overline{A}_+)$, $\{k_n\}_{n=1}^\infty\subset (\ker A_+^*\cap\mathcal{D}(D)^{[\perp]})$, $\left\{\widetilde{k}^{(1)}_n\right\}_{n=1}^\infty\subset(\mathcal{D}(D)//\mathcal{D}_2)$ and $\left\{\widetilde{k}^{(2)}_n\right\}_{n=1}^\infty\subset\mathcal{D}_2$.
At this point it becomes clear that it does not matter which specific decomposition we have chosen in \eqref{eq:decomposition} since any component $(\idty-\mathcal{P})\widetilde{k}^{(1)}$ could be absorbed into $\widetilde{k}^{(2)}$.
For convenience, let us define 
$$v_{F,n}:=(\idty-\mathcal{P})f_n=f_{0,n}+A_{+,F}^{-1}k_n+A_{+,F}^{-1}D\widetilde{k}_n^{(1)}+(\idty-\mathcal{P})\widetilde{k}_n^{(1)}+A_{+,F}^{-1}D\widetilde{k}_n^{(2)}+\widetilde{k}_n^{(2)}\:,$$
from which we get $f_n=v_{F,n}+\mathcal{P}\widetilde{k}_n^{(1)}$, where $\{v_{F,n}\}_{n=1}^\infty\subset\mathcal{D}(V_F^{1/2})$ and $\{\mathcal{P}\widetilde{k}_n^{(1)}\}_{n=1}^\infty\subset\ker V^*$.
Then, we have 
\begin{equation} \label{eq:brennerova}
\lim_{n\rightarrow\infty}\|f_n\|=\lim_{n\rightarrow\infty}\|v_{F,n}+\mathcal{P}\widetilde{k}_n^{(1)}\|=0
\end{equation}
as well as
\begin{align*}
\|f_n&-f_m\|_{\mathfrak{re}_D}^2=\|f_n-f_m\|^2+\|V_F^{1/2}(v_{F,n}-v_{F,m})\|^2\\&+q(A^{-1}_{+,F}D(\widetilde{k}_n^{(1)}-\widetilde{k}^{(1)}_m)+(\widetilde{k}^{(1)}_n-\widetilde{k}^{(1)}_m)+A^{-1}_{+,F}D(\widetilde{k}_n^{(2)}-\widetilde{k}^{(2)}_m)+(\widetilde{k}^{(2)}_n-\widetilde{k}^{(2)}_m))\overset{n,m\rightarrow\infty}{\longrightarrow}0\:,
\end{align*}
which ---  using \eqref{eq:schee} --- simplifies to
\begin{align*}
&\|f_n-f_m\|_{\mathfrak{re}_D}^2\\=&\|f_n-f_m\|^2+\|V_F^{1/2}(v_{F,n}-v_{F,m})\|^2+q(A^{-1}_{+,F}D(\widetilde{k}_n^{(1)}-\widetilde{k}^{(1)}_m)+(\widetilde{k}^{(1)}_n-\widetilde{k}^{(1)}_m))\overset{n,m\rightarrow\infty}{\longrightarrow}0\:.
\end{align*}
Now, since 
\begin{equation*}
\varepsilon\|v_{F,n}-v_{F,m}\|\leq\|V_F^{1/2}(v_{F,n}-v_{F,m})\|\overset{n,m\rightarrow\infty}{\longrightarrow}0
\end{equation*}
we have that $\{v_{F,n}\}_{n=1}^\infty$ converges to an element $v_F\in\mathcal{D}(V_F^{1/2})$. Since $f_n=v_{F,n}+\mathcal{P}\widetilde{k}_n^{(1)}\overset{n\rightarrow\infty}{\longrightarrow}0$, we have that $\{\mathcal{P}\widetilde{k}_n^{(1)}\}_{n=1}^\infty$ converges to $-v_F$. However, since $\mathcal{PD}(D)$ is finite-dimensional, $\{\mathcal{P}\widetilde{k}_n^{(1)}\}_{n=1}^\infty$ converges to an element of $\mathcal{PD}(D)\subset\ker V^*$, from which we get $v_F\in\ker V^*$. But since $\mathcal{D}(V_F^{1/2})\cap\ker V^*=\{0\}$, we get that $v_F=\lim_{n\rightarrow\infty} v_{F,n}=-\lim_{n\rightarrow\infty}\mathcal{P}\widetilde{k}_n^{(1)}=0$. Moreover, the projection $\mathcal{P}$ is injective on $(\mathcal{D}(D)//\mathcal{D}_2)$, which is finite-dimensional. Thus, there exists a number $\varepsilon'>0$ such that 
$$\varepsilon'\|\widetilde{k}_n^{(1)}\|\leq \|\mathcal{P}\widetilde{k}_n^{(1)}\|\overset{n\rightarrow\infty}{\longrightarrow}0\:,$$
which implies that 
\begin{equation}\label{eq:grundlage}
\|\widetilde{k}_n^{(1)}\|\overset{n\rightarrow\infty}{\longrightarrow}0\:.
\end{equation}
Now, since $\dim\mathcal{D}(D)<\infty$, there exists a constant $M<\infty$ such that
\begin{equation} \label{eq:wolga}
q(A_F^{-1}D\widetilde{k}_n^{(1)}+\widetilde{k}_n^{(1)})\leq M\|\widetilde{k}_n^{(1)}\|^2\overset{n\rightarrow\infty}{\longrightarrow}0 \quad\text{by \eqref{eq:grundlage}}\:.
\end{equation}
Altogether, this shows that $\lim_{n\rightarrow\infty}\|f_n\|_{\mathfrak{re}_D}=0$:
\begin{align*}
\|f_n\|_{\mathfrak{re}_D}^2&=\|f_n\|^2+\Real\langle f_n,A_Df_n\rangle=\|v_{F,n}+\mathcal{P}\widetilde{k}_n^{(1)}\|^2+\|V_F^{1/2}v_{F,n}\|^2+q(A_{+,F}^{-1}D\widetilde{k}_n^{(1)}+\widetilde{k}_n^{(1)})\overset{n\rightarrow\infty}{\longrightarrow}0\:,
\end{align*}
where we have used \eqref{eq:brennerova} and \eqref{eq:wolga} as well as the fact $\{v_{F,n}\}_{n=1}^\infty$ is a sequence of elements in $\mathcal{D}(V_F^{1/2})$ that converges to $0$ and that is Cauchy with respect to $\|V_F^{1/2}\cdot\|$, which implies that $V_F^{1/2}v_{F,n}\overset{n\rightarrow\infty}{\longrightarrow}0$ as well. This shows that $\mathfrak{re}_{D,0}$ is closable. This finishes the proof.
\end{proof}
\begin{remark} At this point, we also point out the closability results in \cite[Thm.\ 1]{MMM92} and especially in the very recent publication \cite{DHM}, where similar closability problems for the imaginary part of maximally dissipative relations have been investigated.
\end{remark}
\begin{example}[Continuation of Example \ref{ex:closfailprep}]
\label{ex:closfail}
\normalfont Let us give an example of a dual pair $(C_+,C_-)$ satisfying the assumptions of Theorem \ref{thm:dido} for which there exists a proper maximally accretive extension $C_D$ for which $\mathfrak{re}_{D,0}$ is not closable. Let $\mathcal{H}=L^2(0,1)$ and consider the dual pair of operators
\begin{align*}
C_\pm:\qquad\mathcal{D}(C_\pm)&=\mathcal{C}_c^\infty(0,1), \quad
(C_\pm f)(x)=\pm if''(x)+\frac{\gamma}{x^2}f(x)\:,
\end{align*}
where for simplicity, we choose $\gamma\geq\sqrt{3}$ in order to ensure that $\dim \ker C_+^*=\dim \ker C_-^*=1$. In Example \ref{ex:closfailprep}, we have already shown that $(C_+,C_-)$ is in the extremal case \eqref{eq:extremal}. Moreover, the real part $V$ is just given by the multiplication by the function $\gamma x^{-2}$, with domain $\mathcal{C}_c^\infty(0,1)$. Since $V$ is essentially selfadjoint, we get that $V_F=V_K=\overline{V}$, which is the maximal multiplication operator by the function $\gamma x^{-2}$. This obviously implies that $V_{K}^{1/2}=V_F^{1/2}=\overline{V}^{1/2}$ is the maximal multiplication operator by the function $\sqrt{\gamma}x^{-1}$. Now, observe that the operators $C_\pm$ are obtained from the operators $A_\pm$ in Example \ref{ex:mabaker} via $C_\pm=\mp iA_\pm$. Thus,  it immediately follows that
\begin{align*}
\ker(C_-^*)=\ker(A_-^*)=\spann\{x^{\omega_+}\}\quad\text{and}\quad\ker(C_+^*)=\ker(A_+^*)=\spann\{x^{\overline{\omega_+}}\}\:.
\end{align*}
For the domain of the auxiliary operator $D$ which describes the proper maximally accretive extensions of $(C_+,C_-)$, there are again two possibilities. The choice $\mathcal{D}(D)=\{0\}$ corresponds to the Friedrichs extension $C_{+,F}$ of $C_+$. The second possibility is the choice $\mathcal{D}(D)=\spann\{x^{\omega_+}\}$.
As in Example \ref{ex:mabaker}, any map $D$ from $\spann\{x^{\omega_+}\}$ into $\spann\{x^{\overline{\omega_+}}\}$ has to be of the form
\begin{equation*}
Dx^{\omega_+}=d x^{\overline{\omega_+}}\:,
\end{equation*}
where $d\in\C$. In this case, Theorem \ref{thm:dido} implies that $C_D$ is accretive if and only if
\begin{align} \label{eq:xp}
\Real[x^{\omega_+},Dx^{\omega_+}]&=\Real[x^{\omega_+},dx^{\overline{\omega_+}}]\geq\|\overline{V}^{1/2}x^{\omega_+}\|^2\notag\\
\Leftrightarrow\qquad \Real(d\mu(\omega_+))&\geq\nu(\omega_+)\:,
\end{align}
where the numbers $\mu(\omega_+)$ and $\nu(\omega_+)$ are given by
\begin{align*}
\mu(\omega_+)&=\frac{1}{2\overline{\omega_+}+1}-\frac{2\gamma}{\gamma+i(\overline{\omega_+}+2)(\overline{\omega_+}+1)}\left[\frac{1}{2\overline{\omega_+}+1}-\frac{1}{\omega_++\overline{\omega_+}-1}\right]\\
\nu(\omega_+)&=\frac{\gamma}{\omega_++\overline{\omega_+}-1}\:.
\end{align*}
However, since $x^{\omega_+}\in\mathcal{D}(V_K^{1/2})=\mathcal{D}(\overline{V}^{1/2})=\mathcal{D}(V_F^{1/2})$, this implies by Theorem \ref{thm:closable} that $\mathfrak{re}_{D,0}$ is closable if and only if 
\begin{align*} \label{eq:salvatormundi}
q(C_{+,F}^{-1}Dx^{\omega_+}+x^{\omega_+})&=\Real[x^{\omega_+},Dx^{{\omega_+}}]-\|\overline{V}^{1/2}x^{\omega_+}\|^2=0\\
\Leftrightarrow\qquad\qquad\qquad \Real(d\mu(\omega_+))&=\nu(\omega_+)\:,
\end{align*}
i.e.\ if only if we have equality in \eqref{eq:xp}.

\end{example}
\begin{remark} \normalfont If $A_D$ is sectorial and in the non-extremal case, which means that there exists an $\varepsilon\in(0,2\pi)$ such that the operators $(e^{\pm i\varepsilon}A_D)$ are both accretive, recall that by \cite[Chapter VI, Thm. 1.27]{Kato}, the form $\mathfrak{re}_{D,0}$ is always closable. This is also true for the case $\dim(\mathcal{D}(D))=\infty$ for which the above theorem does not apply. Hence, the only situation for which we do not have a closability result for $\mathfrak{re}_{D,0}$ is when we are in the extremal case \eqref{eq:extremal} {\emph{and}} if we have $\dim(\mathcal{D}(D))=\infty$. \label{rem:closable}
\end{remark}
\section{A partial order on the real parts} \label{sec:order}
If $\mathfrak{re}_{D,0}$ is closable, there exists a selfadjoint operator $V_D$ associated to the closure $\mathfrak{re}_D$. Let us now show that this operator is an extension of $V$:
\begin{lemma} Let $(A_+,A_-)$ be a dual pair that satisfies the assumptions of Theorem \ref{thm:dido}. Moreover, assume that $\mathfrak{re}_{D,0}$ is closable. Then, $V\subset V_D$. \label{lemma:brexit}
\end{lemma}
\begin{proof} Observe that for any $f,g\in\mathcal{D}(A_+)=\mathcal{D}(V)=\mathcal{D}(\mathfrak{re}_{D,0})\subset\mathcal{D}(V_D)$, we have
\begin{equation*}
\langle f,Vg\rangle=\mathfrak{re}_{D,0}(f,g)=\mathfrak{re}_D(f,g)=\langle f,V_Dg\rangle\:.
\end{equation*}
Consequently, we get $\langle f,(V-V_D)g\rangle=0$, which implies $V_Dg=Vg$ since $f$ was an arbitrary element of the dense set $\mathcal{D}(V)$.
\end{proof} 
Next, let us determine the form domain of $V_D$. Using that $V_D$ is a non-negative selfadjoint extension of $V$, we know by \eqref{eq:fussball} that $\mathcal{D}(V_F^{1/2})\subset\mathcal{D}(V_D^{1/2})$ and by \eqref{eq:alonsosimon} that there exists a subspace $\mathcal{M}\subset\ker V^*$ such that $\mathcal{D}(V_D^{1/2})=\mathcal{D}(V_F^{1/2})\dot{+}\mathcal{M}$. In the case that $\mathcal{PD}(D)$ is finite-dimensional, we will show that $\mathcal{M}=\mathcal{PD}(D)$, i.e.\ the part of $\mathcal{D}(D)$ that can be projected onto $\ker V^*$.
\begin{lemma} Let $V$ be strictly positive, i.e. $V\geq\varepsilon>0$ and assume that $\mathfrak{re}_{D,0}$ is closable. Then, the domain of $V_D^{1/2}$ is given by
\begin{equation*}
\mathcal{D}(V_D^{1/2})=\overline{\mathcal{D}(V_F^{1/2})\dot{+}\mathcal{P}\mathcal{D}(D)}^{\|\cdot\|_{\mathfrak{re}_D}}\:.
\end{equation*}
In particular, if $\dim\left(\mathcal{PD}(D)\right)<\infty$, we get
$$\mathcal{D}(V_D^{1/2})=\mathcal{D}(V_F^{1/2})\dot{+}\mathcal{P}\mathcal{D}(D)\:.$$
\label{lemma:sts}
\end{lemma}
\begin{proof}
By Theorem \ref{prop:festnumrange}, we have that $\mathcal{D}(A_F)\subset\mathcal{D}(V_F^{1/2})$, which implies that any element of $$\mathcal{D}(A_D)=\mathcal{D}(A)\dot{+}\{A_F^{-1}D\widetilde{k}+\widetilde{k}:\widetilde{k}\in\mathcal{D}(D)\}\dot{+}\{A_F^{-1}k:k\in\ker A^*\cap\mathcal{D}(D)^{[\perp]}\}$$ can be written as
\begin{equation*}
f+A_F^{-1}D\widetilde{k}+\widetilde{k}+A_F^{-1}k=\underbrace{(f+A_F^{-1}D\widetilde{k}+(\idty-\mathcal{P})\widetilde{k}+A_F^{-1}k)}_{\in\mathcal{D}(V_F^{1/2})}+\underbrace{\mathcal{P}\widetilde{k}}_{\in\mathcal{PD}(D)}\:,
\end{equation*}
which implies that $\mathcal{D}(A_D)\subset\mathcal{D}(V_F^{1/2})\dot{+}\mathcal{PD}(D)$. On the other hand, we have by Lemma \ref{lemma:brexit} that $V_D$ is a positive selfadjoint extension of $V$, from which we get by \cite{Alonso-Simon} that $\mathcal{D}(V_F^{1/2})\subset\mathcal{D}(V_D^{1/2})$. As any $\mathcal{P}\widetilde{k}\in\mathcal{PD}(D)$ can be written as
$$\mathcal{P}\widetilde{k}=\underbrace{(A_F^{-1}D\widetilde{k}+\widetilde{k})}_{\in\mathcal{D}(A_D)}-\underbrace{(A_F^{-1}D\widetilde{k}+(\idty-\mathcal{P})\widetilde{k})}_{\in\mathcal{D}(V_F^{1/2})}$$
and since $\mathcal{D}(A_D)\subset\mathcal{D}(V_D^{1/2})$ and $\mathcal{D}(V_F^{1/2})\subset\mathcal{D}(V_D^{1/2})$, this implies that $\mathcal{P}\widetilde{k}\in\mathcal{D}(V_D^{1/2})$, and thus $\mathcal{PD}(D)\subset\mathcal{D}(V_D^{1/2})$. Consequently, we have $\mathcal{D}(A_D)\subset\mathcal{D}(V_F^{1/2})\dot{+}\mathcal{PD}(D)\subset\mathcal{D}(V_D^{1/2})$ and since 
$\overline{\mathcal{D}(A_D)}^{\|\cdot\|_{\mathfrak{re}_D}}=\mathcal{D}(V_D^{1/2})$, we get
$$ \mathcal{D}(V_D^{1/2})=\overline{\mathcal{D}(A_D)}^{\|\cdot\|_{\mathfrak{re}_D}}\subset \overline{\mathcal{D}(V_F^{1/2})\dot{+}\mathcal{P}\mathcal{D}(D)}^{\|\cdot\|_{\mathfrak{re}_D}}\subset \mathcal{D}(V_D^{1/2})\:, $$
which proves the first assertion of the lemma. Next, let us show that $\mathcal{D}(V_F^{1/2})$ is a closed subspace of $\mathcal{D}(V_D^{1/2})$ with respect to $\|\cdot\|_{\mathfrak{re}_D}$. This follows from the fact that for any $f\in\mathcal{D}(V)$ we get that
\begin{equation*}
\|f\|_{\mathfrak{re}_D}^2=\|f\|^2+\Real\langle f,A_Df\rangle=\|f\|^2+\langle f,Vf\rangle\:,
\end{equation*}
which means that $\overline{\mathcal{D}(V)}^{\|\cdot\|_{\mathfrak{re}_D}}=\mathcal{D}(V_F^{1/2})$. Since $\dim(\mathcal{PD}(D))<\infty$, we have by \cite[Problem 13]{Halmos} that $\mathcal{D}(V_F^{1/2})\dot{+}\mathcal{PD}(D)$ is a closed subspace of $\mathcal{D}(V_D^{1/2})$ with respect to the $\|\cdot\|_{\mathfrak{re}_D}$-norm and by what we have shown before, this yields
\begin{equation*}
\mathcal{D}(V_D^{1/2})=\overline{\mathcal{D}(V_F^{1/2})\dot{+}\mathcal{PD}(D)}^{\|\cdot\|_{\mathfrak{re}_D}}=\mathcal{D}(V_F^{1/2})\dot{+}\mathcal{PD}(D)\subset\mathcal{D}(V_D^{1/2})\:,
\end{equation*}
which is the desired result.
\end{proof}
Finally, let us determine the action of $\mathfrak{re}_D$.
\begin{theorem} \label{thm:townsende} Let $V$ be as in Lemma \ref{lemma:sts} and moreover, assume that $$\dim\mathcal{PD}(D)<\infty\:.$$ Then, there exists a non-negative selfadjoint operator $B$ with $\mathcal{D}(B)=\mathcal{PD}(D)$ such that for any $v_F\in\mathcal{D}(V_F^{1/2})$ and any $\eta\in\mathcal{PD}(D)$, we have
\begin{equation*}
\mathfrak{re}_D(v_F+\eta)=\|V_D^{1/2}(v_F+\eta)\|^2=\|V_F^{1/2}v_F\|^2+q_B(\eta)\:,
\end{equation*}
where $q_B$ denotes the quadratic form associated to $B$. It is given by
\begin{equation*}
q_B(\eta)=\Real[\mathcal{P}^{-1}\eta,D\mathcal{P}^{-1}\eta]-\|V_K^{1/2}\mathcal{P}^{-1}\eta\|^2\:.
\end{equation*}
Here, $\mathcal{P}^{-1}$ denotes the inverse of $\mathcal{P}$ restricted to a subspace of $\mathcal{D}(D)$ that is complementary to $\mathcal{D}(D)\cap\mathcal{D}(V_F^{1/2})$. The form $q_B$ does not depend on the specific choice of this subspace. Moreover, if we choose $\{\eta_i\}_{i=1}^n$ to be an orthonormal basis of $\mathcal{PD}(D)$, the elements of the matrix representation of $B$ with respect to $\{\eta_i\}_{i=1}^n$ are given by
\begin{align*}
B&=(b_{ij})_{i,j=1}^n\quad\text{where}\\ b_{ij}&=\frac{1}{2}\left([\mathcal{P}^{-1}\eta_i,D\mathcal{P}^{-1}\eta_j]+[D\mathcal{P}^{-1}\eta_i,\mathcal{P}^{-1}\eta_j]\right)-\langle V_K^{1/2}\mathcal{P}^{-1}\eta_i,V_K^{1/2}\mathcal{P}^{-1}\eta_j\rangle\:.
\end{align*}
\end{theorem}
\begin{proof} For any $\eta\in\mathcal{PD}(D)$, there exists a $\widetilde{k}\in\mathcal{D}(D)$ such that $\eta=\mathcal{P}\widetilde{k}$.  In the case that $\mathcal{D}(V_F^{1/2})\cap\mathcal{D}(D)$ is non-trivial, which means that $\ker(\mathcal{P})\cap\mathcal{D}(D)$ is non-trivial, this choice of $\widetilde{k}$ is not unique as for any $\chi\in\mathcal{D}(D)\cap\mathcal{D}(V_F^{1/2})$ we would still have that $\mathcal{P}(\widetilde{k}+\chi)=\eta$. However, if we choose a subspace of $\mathcal{S}\subset\mathcal{D}(D)$ that is complementary to $\mathcal{D}(V_F^{1/2})\cap\mathcal{D}(D)$ in $\mathcal{D}(D)$, then we can define the inverse of $\mathcal{P}$ on $\mathcal{S}$:
\begin{align*}
\mathcal{P}^{-1}: \mathcal{D}(\mathcal{P}^{-1})&=\mathcal{PS}\\
                       \mathcal{P}\widetilde{k}&\mapsto\widetilde{k},\quad\widetilde{k}\in\mathcal{S}\:.
\end{align*}
 Now, for any $v_F\in\mathcal{D}(V_F^{1/2})$ and $\widetilde{k}\in\mathcal{S}$, let us pick a sequence $\{f_n\}_{n=1}^\infty\subset\mathcal{D}(V)$ such that
\begin{align*}
\|\cdot\|_{\mathfrak{re}_D}-\lim_{n\rightarrow\infty} f_n=\left[v_F-A_F^{-1}D\widetilde{k}-(\idty-\mathcal{P})\widetilde{k}\right]\in\mathcal{D}(V_F^{1/2})\:,
\end{align*}
where $``\|\cdot\|_{\mathfrak{re}_D}-\lim"$ denotes the limit with respect to the $\|\cdot\|_{\mathfrak{re}_D}$-norm.
This implies that the sequence $g_n:=f_n+A_F^{-1}D\widetilde{k}+\widetilde{k}$ converges to $v_F+\mathcal{P}\widetilde{k}$ in the usual norm $\|\cdot\|$. Next, let us show that $\{g_n\}_{n=1}^\infty$ is Cauchy with respect to $\|\cdot\|_{\mathfrak{re}_D}$:
\begin{align*}\|g_n-g_m\|^2_{\mathfrak{re}_D}=\|f_n-f_m\|^2+\|V_F^{1/2}(f_n-f_m)\|^2=\|f_n-f_m\|_{\mathfrak{re}_D}^2\overset{n,m\rightarrow\infty}{\longrightarrow}0\:,
\end{align*}
since $\{f_n\}_{n=1}^\infty$ has a limit with respect to $\|\cdot\|_{\mathfrak{re}_D}$. Thus, we get
\begin{align*}
\|v_F+\mathcal{P}\widetilde{k}\|^2_{\mathfrak{re}_D}&=\lim_{n\rightarrow\infty}\|g_n\|^2_{\mathfrak{re}_D}\\&=\lim_{n\rightarrow\infty}(\|g_n\|^2+\|V_F^{1/2}(f_n+A_F^{-1}D\widetilde{k}+(\idty-\mathcal{P})\widetilde{k})\|^2+q(A_F^{-1}D\widetilde{k}+\widetilde{k}))\\&=\|v_F+\mathcal{P}\widetilde{k}\|^2+\|V_F^{1/2}v_F\|^2+q(A_F^{-1}D\widetilde{k}+\widetilde{k})\\
&=\|v_F+\mathcal{P}\widetilde{k}\|^2+\|V_F^{1/2}v_F\|^2+\Real[\widetilde{k},D\widetilde{k}]-\|V_K^{1/2}\widetilde{k}\|^2
\end{align*}
and since for any $\phi\in\mathcal{D}(\mathfrak{re}_D)$ we have $\|\phi\|_{\mathfrak{re}_D}^2=\|\phi\|^2+\mathfrak{re}_D(\phi)$, this allows us to read off
\begin{equation*}
\mathfrak{re}_D(v_F+\mathcal{P}\widetilde{k})=\|V_F^{1/2}v_F\|^2+\Real[\widetilde{k},D\widetilde{k}]-\|V_K^{1/2}\widetilde{k}\|^2\:.
\end{equation*}
However, if $\mathcal{D}(V_F^{1/2})\cap\mathcal{D}(D)$ is non-trivial, we could have added a $\chi\in(\mathcal{D}(V_F^{1/2})\cap\mathcal{D}(D))$ such that $\mathcal{P}(\widetilde{k}+\chi)=\mathcal{P}{\widetilde{k}}=\eta$. But since $\mathfrak{re}_{D,0}$ was assumed to be closable, we have by Theorem \ref{thm:closable} that $q(A_F^{-1}D\chi+\chi)=0$ and by \eqref{eq:schee}, we have in addition that 
\begin{equation} \label{eq:foermchen}
q(A_F^{-1}D\widetilde{k}+\widetilde{k}+A_F^{-1}D\chi+\chi)=q(A_F^{-1}D\widetilde{k}+\widetilde{k})\:.
\end{equation}
Thus, the specific choice of $\mathcal{S}\subset\mathcal{D}(D)$ --- as long as it is complementary to $(\mathcal{D}(D)\cap\mathcal{D}(V_F^{1/2}))$ in $\mathcal{D}(D)$ --- does not affect the value of $\mathfrak{re}_D(v_F+\eta)$, where $v_F\in\mathcal{D}(V_F^{1/2})$ and $\eta\in\mathcal{PD}(D)$, where $\eta=\mathcal{P}\widetilde{k}$ for a unique $\widetilde{k}\in\mathcal{S}$. Hence, by Equation \eqref{eq:foermchen}, we get
\begin{equation} \label{eq:altus}
\mathfrak{re}_D(v_F+\eta)=\mathfrak{re}_D(v_F+\mathcal{P}\widetilde{k}_\eta)=\|V_F^{1/2}v_F\|^2+\Real[\widetilde{k}_\eta,D\widetilde{k}_\eta]-\|V_K^{1/2}\widetilde{k}_\eta\|^2\:,
\end{equation} 
where $\widetilde{k}_\eta$ is the unique element of $\mathcal{S}$ such that $\mathcal{P}\widetilde{k}_\eta=\eta$, or in other words, we get $\widetilde{k}_\eta=\mathcal{P}^{-1}\eta$. Plugged into \eqref{eq:altus}, this yields
\begin{equation} \label{eq:hirnschmarrn}
\mathfrak{re}_D(v_F+\eta)=\|V_F^{1/2}v_F\|^2+\Real[\mathcal{P}^{-1}\eta,D\mathcal{P}^{-1}\eta]-\|V_K^{1/2}\mathcal{P}^{-1}\eta\|^2\:.
\end{equation}
Now, since we have shown in Lemma \ref{lemma:brexit} that $V_D$ is a non-negative selfadjoint extension of $V$, we know by \cite{Alonso-Simon} that there exists a subspace $\mathcal{D}(B)\subset\ker V^*$ and a non-negative auxiliary operator $B$ from $\mathcal{D}(B)$ into $\mathcal{D}(B)$ such that 
\begin{equation} \label{eq:oberschmarrn}
\mathfrak{re}_D(v_F+\eta)=\|V_D^{1/2}(v_F+\eta)\|^2=\|V_F^{1/2}v_F\|^2+q_B(\eta)\:,
\end{equation}
where $v_F\in\mathcal{D}(V_F^{1/2})$ and $\eta\in\mathcal{D}(B)$. The form $q_B$ is given by $q_B(\eta)=\langle \eta,B\eta\rangle$ for all $\eta\in\mathcal{D}(B)$.\footnote{Note that we are only considering the finite-dimensional case, which means that we do not have to worry about closures and domains.} Comparing Equations \eqref{eq:oberschmarrn} and \eqref{eq:hirnschmarrn}, we can read off that
\begin{equation} \label{eq:megaschmarrn}
q_B(\eta)=\Real[\mathcal{P}^{-1}\eta,D\mathcal{P}^{-1}\eta]-\|V_K^{1/2}\mathcal{P}^{-1}\eta\|^2\:,
\end{equation}
which is the desired result. To determine the entries of the non-negative matrix $(b_{ij})_{ij}$, we use that the sesquilinear form $q_B(\cdot,\cdot)$ associated to $q_B$ is given by
\begin{equation*}
q_B(\eta_i,\eta_j)=\frac{1}{2}([\mathcal{P}^{-1}\eta_i,D\mathcal{P}^{-1}\eta_j]+[D\mathcal{P}^{-1}\eta_i,\mathcal{P}^{-1}\eta_j])-\langle V_K^{1/2}\mathcal{P}^{-1}\eta_i,V_K^{1/2}\mathcal{P}^{-1}\eta_j\rangle\:.
\end{equation*}
This immediately follows from the fact that $\mathfrak{re}_D(\eta,\eta)=\mathfrak{re}_D(\eta)$, which can be seen by direct inspection. Now, since $b_{ij}=\langle \eta_i,B\eta_j\rangle=q_B(\eta_i,\eta_j)$, this finishes the proof.
\end{proof}
\begin{remark} Note that we are not assuming that $\dim(\mathcal{D}(D))<\infty$ for Theorem \ref{thm:townsende} and the following Corollary \ref{coro:order}. As long as $\mathfrak{re}_{D,0}$ is closable (cf. Remark \ref{rem:closable} for sufficient conditions for this to be true), we only need to assume that $\dim(\mathcal{P}\mathcal{D}(D))<\infty$. We hope to be able to remove this technical assumption in a future work.
\end{remark}
The previous result allows us to deduce a way of comparing the real parts $V_{D_1}$ and $V_{D_2}$ of two different extensions $A_{D_1}$ and $A_{D_2}$:
\begin{corollary} \label{coro:order}
Let $D_1$ and $D_2$ parametrize two different proper maximally dissipative extensions of $(A_+,A_-)$ and let $B_1$ and $B_2$ be the two associated non-negative auxiliary operators whose quadratic forms are given in \eqref{eq:megaschmarrn}.
Then $B_1\geq B_2$ if and only if $\mathcal{PD}(D_1)\subset\mathcal{PD}(D_2)$ and 
\begin{equation} \label{eq:nerventeil}
\Real [\widetilde{k},D_1\widetilde{k}]\geq \Real [\widetilde{k}, D_2\widetilde{k}]
\end{equation}
for all $\widetilde{k}\in\mathcal{D}(D_1)$.
\end{corollary}
\begin{proof} By definition, $B_1\geq B_2$ as operators on a finite-dimensional space if and only if $\mathcal{D}(B_1)\subset\mathcal{D}(B_2)$ and $q_{B_1}(\eta)\geq q_{B_2}(\eta)$ for all $\eta\in\mathcal{D}(B_1)$. Since $\mathcal{D}(B_{1,2})=\mathcal{PD}(D_{1,2})$, this shows the first condition of the corollary. Now, for any $\eta\in\mathcal{D}(B_1)=\mathcal{PD}(D_1)$ we have by \eqref{eq:megaschmarrn}
\begin{align}
q_{B_1}(\eta)&-q_{B_2}(\eta)\notag\\&=\Real [\mathcal{P}^{-1}\eta,D_1\mathcal{P}^{-1}\eta]-\|V_K^{1/2}\mathcal{P}^{-1}\eta\|^2-(\Real [\mathcal{P}^{-1}\eta,D_2\mathcal{P}^{-1}\eta]-\|V_K^{1/2}\mathcal{P}^{-1}\eta\|^2)\notag\\&=\Real [\mathcal{P}^{-1}\eta,D_1\mathcal{P}^{-1}\eta]-\Real [\mathcal{P}^{-1}\eta,D_2\mathcal{P}^{-1}\eta]\:,
\end{align}
which is non-negative for all $\eta\in\mathcal{D}(B_1)$ if and only if $$\Real[\mathcal{P}^{-1}\eta,D_1\mathcal{P}^{-1}\eta]\geq\Real [\mathcal{P}^{-1}\eta,D_2\mathcal{P}^{-1}\eta]$$ for all $\eta\in\mathcal{D}(B_1)=\mathcal{PD}(D_1)$. Thus, Condition \eqref{eq:nerventeil} being satisfied is sufficient for $B_1\geq B_2$. Let us now show that it is also necessary. Assume that there exists a $\widetilde{k}\in\mathcal{D}(D_1)$ such that
\begin{equation*}
\Real[\widetilde{k},D_1\widetilde{k}]<\Real[\widetilde{k},D_2\widetilde{k}]\:\:\Leftrightarrow\:\: \Real[\widetilde{k},D_1\widetilde{k}]-\|V_K^{1/2}\widetilde{k}\|^2<\Real[\widetilde{k},D_2\widetilde{k}]-\|V_K^{1/2}\widetilde{k}\|^2\:.
\end{equation*}
Observe that by Theorem \ref{thm:closable}, this means that $\widetilde{k}\in(\mathcal{D}(V_F^{1/2})\cap\mathcal{D}(D_1))$ is not possible in this case, since this would imply that $\Real[\widetilde{k},D_2\widetilde{k}]-\|V_K^{1/2}\widetilde{k}\|^2=0$, but by accretivity of $A_{D_1}$ we have by virtue of Theorem \ref{thm:dido} that $\Real[\widetilde{k},D_1\widetilde{k}]-\|V_K^{1/2}\widetilde{k}\|^2\geq 0$. Hence, $0\neq\mathcal{P}\widetilde{k}=:\eta\in\mathcal{PD}(D_1)$ or $\mathcal{P}^{-1}\eta=\widetilde{k}$. Therefore we get
\begin{equation*}
q_{B_1}(\eta)-q_{B_2}(\eta)=\Real[\widetilde{k},D_1\widetilde{k}]-\Real[\widetilde{k},D_2\widetilde{k}]<0\:,
\end{equation*}
which shows that $B_1\not\geq B_2$ if Condition \eqref{eq:nerventeil} is not satisfied, which therefore is necessary for $B_1\geq B_2$ to be true.
This shows the corollary.
\end{proof}
\begin{remark} \normalfont These results allow us to give first estimates of the lower bound of the real part. From \cite[Thm. 2.13]{Alonso-Simon}, it follows that 
\begin{equation}
\frac{\alpha\delta}{1+\delta}\leq \inf_{0\neq\psi\in\mathcal{D}(A_D)}\frac{\Real\langle \psi,A_D\psi\rangle}{\|\psi\|^2}\leq \alpha\delta\:,
\end{equation}
where $\alpha$ is the lower bound of the real part of $A$ and $\delta$ is the lower bound of the quadratic form $q_B$:
$$\alpha:=\inf\left\{\frac{\Real\langle \psi,A\psi\rangle} {\|\psi\|^2}: \psi\in\mathcal{D}(A), \psi\neq 0\right\} \:\:\:\: \text{and} \:\:\:\: \delta:=\inf\left\{\frac{q_B(\eta)}{\|\eta\|^2}:\eta\in\mathcal{PD}(D), \eta\neq 0\right\}\:.$$ As mentioned in \cite[Thm. 2.13]{Alonso-Simon}, this means in particular that $$\inf\left\{\frac{\Real\langle \psi,A_D\psi\rangle} {\|\psi\|^2}: \psi\in\mathcal{D}(A_D), \psi\neq 0\right\}=0$$ if and only if $\delta=0$. (Recall that we have assumed that $\alpha\geq \varepsilon >0$.)
\end{remark}
\begin{example}[Continuation of Example \ref{ex:mabaker}] \normalfont
Consider the dual pair $(A_+,A_-)$ as defined in Example \ref{ex:mabaker}. Now, since $\mathcal{D}(D)=\spann\{x^{\omega_+}\}$ and $$x^{\omega_+}=\underbrace{(x^{\omega_+}-x)}_{\in\mathcal{D}(V_F^{1/2})}+\underbrace{x}_{\in\ker V^*}\:,$$ we get $\mathcal{P}\mathcal{D}(D)=\spann\{x\}$. In particular, we have that $\mathcal{D}(D)\cap\mathcal{D}(V_F^{1/2})=\{0\}$ from which we see by virtue of Theorem \ref{thm:closable} that $\mathfrak{re}_{D,0}$ is closable. Moreover, the operator $\mathcal{P}\upharpoonright_{\mathcal{D}(D)}$ is injective and thus, we define $\mathcal{P}^{-1}x=x^{\omega_+}$. By Lemma \ref{lemma:sts}, the associated operator $V_D$ has form domain
\begin{equation*}
\mathcal{D}(V_D^{1/2})=\mathcal{D}(V_F^{1/2})\dot{+}\mathcal{P}\spann\{x^{\omega_+}\}=\mathcal{D}(V_F^{1/2})\dot{+}\spann\{x\}=H^1_0(0,1)\dot{+}\spann\{x\}
\end{equation*}
and the quadratic form acts like
\begin{align*}
\|{V_D}^{1/2}(f+\lambda x)\|^2&=\|V_F^{1/2}f\|^2+|\lambda|^2\left(\Real[\mathcal{P}^{-1}x,D\mathcal{P}^{-1}x]-\|V_K^{1/2}\mathcal{P}^{-1}x\|^2\right)\\
&=\|V_F^{1/2}f\|^2+|\lambda|^2\left(\Real(d\sigma(\omega_+))-\tau(\omega_+)\right)\:,
\end{align*}
where $f\in\mathcal{D}(V_F^{1/2})$. Recall that the numbers $\sigma(\omega_+)$ and $\tau(\omega_+)$ have been defined in Equations \eqref{eq:thomismus} and \eqref{eq:bonaventura}.
The operator $B_D$ associated to the quadratic form is a map from $\spann\{x\}$ to $\spann\{x\}$ and is therefore of the form $B_Dx=bx$, where $b\in\C$. By Theorem \ref{thm:townsende}, we have that $b$ is given by
$$b=\Real[\mathcal{P}^{-1}(\sqrt{3}x),D\mathcal{P}^{-1}(\sqrt{3}x)]-\|V_K^{1/2}\mathcal{P}^{-1}\sqrt{3}x\|^2=3(\Real(d\sigma(\omega_+))-\tau(\omega_+))\:,$$
where the factor $\sqrt{3}$ comes from normalizing the function $x$. Now, for two different maximally accretive extensions $A_{D_1}$ and $A_{D_2}$, we have that if $\Real(d_1\sigma(\omega_+))=\Real [\widetilde{k},D_1\widetilde{k}]\geq\Real(d_2\sigma(\omega_+))=\Real [\widetilde{k},D_2\widetilde{k}]$, this implies that $B_{D_1}\geq B_{D_2}$.

Finally, let us construct the selfadjoint operators $V_D$ using the Birman--Kre\u\i n--Vishik theory for positive symmetric operators. For $\mathcal{D}(D)=\{0\}$ we get the Friedrichs extension $V_F$ of $V$. The other possibility is that $\mathcal{D}(D)=\spann\{x^{\omega_+}\}$ with $\mathcal{P}\spann\{x^{\omega_+}\}=\spann\{x\}$. We then get
\begin{align*}
V_D:\:\mathcal{D}(V_D)&=\mathcal{D}(V)\dot{+}\spann\{V_F^{-1}B_Dx+x\}\dot{+}\spann\{V_F^{-1}(2-3x)\}\\
&=\mathcal{D}(V)\dot{+}\spann\left\{3[\Real(d\sigma(\omega_+))-\tau(\omega_+)]V_F^{-1}x+x\right\}\dot{+}\spann\{V_F^{-1}(2-3x)\}\\
f&\mapsto-f''\:,
\end{align*}
where the last span comes from the fact that $(2-3x)\perp x$. Also, note that it is not difficult to compute $V_F^{-1}1$ and $V_F^{-1}x$:

\begin{equation*}
V_F^{-1}1=\frac{x^2-x}{2}\quad\text{and}\quad
V_F^{-1}x=\frac{x^3-x}{6}\:.
\end{equation*}
\end{example}
\subsection*{Acknowledgements} I am very grateful to Yu.\ {Arlinski\u\i} and M.\ M.\ Malamud for providing me with useful references. Parts of this work have been done during my PhD studies at the University of Kent in Canterbury, UK (cf.\ \cite[Chapter 8]{thesis}). Thus, I would like to thank my thesis advisors Sergey Naboko and Ian Wood for support, guidance and useful remarks on this paper. Moreover, I am very grateful to the UK Engineering and Physical Sciences Research Council (Doctoral Training Grant Ref.\ EP/K50306X/1) and the School of Mathematics, Statistics and Actuarial Science at the University of Kent for a PhD studentship.

\end{document}